\documentclass[12pt]{amsart}
\usepackage[margin=0.8in]{geometry}
\usepackage{amsfonts}
\usepackage{amsthm}
\usepackage{amssymb}
\usepackage{times} 
\usepackage{graphicx}
\usepackage{hyperref}
\usepackage{tikz}
\usepackage{comment}
\usetikzlibrary{arrows}
\usepackage[all]{xy}
\usepackage{tikz-cd}
\usepackage{lipsum}
\usepackage{mathtools}
\let\oldtocsection=\tocsection

\let\oldtocsubsection=\tocsubsection

\let\oldtocsubsubsection=\tocsubsubsection

\renewcommand{\tocsection}[2]{\hspace{0em}\oldtocsection{#1}{#2}}
\renewcommand{\tocsubsection}[2]{\hspace{1em}\oldtocsubsection{#1}{#2}}
\renewcommand{\tocsubsubsection}[2]{\hspace{2em}\oldtocsubsubsection{#1}{#2}}

\date{}

\numberwithin{equation}{section}
\newtheorem{thm}{Theorem}[section]
\newtheorem{prop}[thm]{Proposition}

\tikzset{node distance=5cm, auto}

\newtheorem{lemma}[thm]{Lemma}

\newtheorem{defn}[thm]{Definition}
\newtheorem{cor}[thm]{Corollary}
\newtheorem{remark}[thm]{Remark}
\newtheorem{example}[thm]{Example}
\newtheorem{question}[thm]{Question}
\makeatletter
\newcommand{\etale}{\'etal\@ifstar{\'e}{e\xspace}}
\makeatother
\usepackage[OT2,T1]{fontenc}
\DeclareSymbolFont{cyrletters}{OT2}{wncyr}{m}{n}
\DeclareMathSymbol{\Sha}{\mathalpha}{cyrletters}{"58}

\begin{document}

\title{The genus of division algebras over discrete valued fields }
\author{S. Srimathy}
\address{School of Mathematics\\
Tata Institute of Fundamental Research\\
Mumbai \\
India}
\email{srimathy@math.tifr.res.in}

\begin{abstract}
Given a field with a set of discrete valuations $V$, we show how the genus of a division algebra over the field  is related to the genus of the residue algebras at various valuations in $V$ and the ramification data. When the division algebra is a quaternion, we show the triviality of genus over many fields  which include higher local fields, function fields of curves over higher local fields  and function fields of curves over real closed fields.  We also consider  function fields of  curves over global fields with a rational point  and show how the genus problem is related to the $2$-torsion of the Tate-Shafarevich  group of its Jacobian. As a special case, we show how the methods developed yield better bounds on the size of the genus over function fields of  elliptic curves and demonstrate how they can be computed  directly using arithmetic data of the elliptic curve  with a number of examples.
\end{abstract}
\maketitle

\tableofcontents
\section{Introduction}\label{sec:intro}
 Given a finite dimensional division algebra $D$ over $K$, one can ask  how much of information is carried by the collection of  maximal subfields that $D$ contains. Let $[D]$ denote the class of $D$ in $Br(K)$.  Then the \textit{genus} of $[D]$, denoted $gen([D])$  is the collection of all classes in $Br(K)$ whose underlying division algebras share the the same  maximal subfields as $D$. One can  ask if this collection is finite and if so what is its size. This has been the studied extensively over the past few years (see \cite{rapinchuk_russian}, \cite{rapinchuk_genus_unramified}, \cite{size_genus}, \cite{RR_genus}, \cite{recent_genus}) and plays an important role in the analysis of weak commensurability of Zariski-dense subgroups of the corresponding algebraic group (\cite[Remark 5.4]{pr2}, \cite[\S6]{pr1}).   \\
 \indent Note that if $[D]$ is of exponent $n\geq 3$ in $Br(K)$,  for every $i$ coprime to $n$, the underlying  division algebra  of $D^{\otimes i}$ has the same collection of   maximal subfields as $D$.  In particular, the  $gen([D])$ contains classes other than $[D]$ unless $[D]$ is of exponent $2$.  In the case of finitely generated fields of characteristic coprime to the exponent of $[D]$, it shown in \cite{RR_genus}  that $gen([D])$ is always finite and explicit upper bounds on the size of the genus are given for many cases in \cite{size_genus}. When $D$ is a quaternion algebra or in general  an exponent $2$ division algebra, a natural question is to know if  $gen([D])$ is always trivial i.e is a singleton set. In other words, we may ask if  quaternion algebras are completely determined by the collection of maximal subfields contained in them. Although the answer is no in general, the non-triviality of the genus for quaternions  is known to happen over very large fields of infinite transcendence degree (\cite[\S 2]{garibaldi_saltman}). So one may wonder if it is trivial for reasonably nice fields.  The  answer to this question is unknown although this is known to be true for some  fields. Examples include  global fields (\cite[\S 3.6]{rapinchuk_genus_unramified}) and more generally for \textit{transparent fields}  (\S6 in \cite{garibaldi_saltman}). Several other examples can be found in \cite{recent_genus}. Moreover it is shown in \cite[Theorem 3.5]{rapinchuk_genus_unramified} that  when $char~K \neq 2$, the property of genus being trivial for exponent $2$ algebras is stable under purely transcendental extensions.   \\
 \indent An important tool used in the literature so far for genus computations is the \textit{unramified Brauer group} of $K$ with respect to a set of discrete valuations $V$, denoted by $Br(K)_{V}$. For example, it is shown in \cite[Theorem 2.2]{rapinchuk_genus_unramified} that if $V$ satisfies certain conditions and if  the $n$-torsion component ${}_nBr(K)_{V}$ is finite then for any division algebra $D$ of exponent and degree $n$, its genus is upper bounded by $|{}_nBr(K)_{V}| \cdot \phi(n)^r$ where $\phi$ is the Euler's Totient function and $r$ is the number of ramification places of $D$. While this bound is extremely useful in showing finiteness of genus and in bounding its size, it does not give any explicit description of the elements in the genus.  Moreover, often times the bound given by ${}_nBr(K)_{V}$ is loose and sometimes ${}_nBr(K)_{V}$ is not finite for some fields that arise naturally in arithmetic geometry.  For example, take $K = \mathbb{Q}(\!(t)\!)$ with $V$ being  the $t$-adic valuation. Then by \cite[ Chapter XXII, Theorem 2]{serre_local}) we have $Br(\mathbb{Q}) \simeq Br(\mathbb{Q}(\!(t)\!))_{V} $ and hence $ {}_nBr(K)_{V}$ is not finite. Another example is when $K = k(\!(t)\!)(C)$ and $V$ is the set of discrete valuations arising from codimension one points of some  regular proper model of $C$ over $k[[t]]$. Then $Br(k) \subseteq Br(K)_{V}$ and therefore $ {}_nBr(K)_{V}$ is not finite if ${}_nBr(k)$ is not finite (this happens for example when $k= \mathbb{Q}$).  \\
\indent  In this paper, we define the \textit{separable genus} of $[D]$, denoted $gen^s([D])$ to be  collection of all classes of division algebras  in $Br(K)$  whose underlying division algebras share the same \textit{separable}  maximal subfields as $D$.  For simplicity,  instead of using the term "separable genus", we simply refer to it as the "genus".  We include separability assumption in our definition of genus to be consistent with the definition of genus of the corresponding algebraic group $SL_{1,D}$ (see Remark 5.1 in \cite{rapinchuk_genus_unramified}) and also to be consistent with the notion defined in \cite{recent_genus}.  Note that when $char~K$ is zero or is coprime to the degree $n$ of $D$ (which is assumed for the fields considered in \cite{rapinchuk_genus_unramified}), every maximal subfield of $D$ is separable over $K$ and the two notions coincide i.e, $gen^s([D]) = gen([D])$.  We study the genus of  tame division algebras over discrete valued fields that are not necessarily finitely generated. We also do not assume that ${}_nBr(K)_V$ is finite. We show how the genus  over these fields is related to  the genus of  residue  algebras and the ramification data.  In many cases, we give explicit description of the elements in the genus and show that we get better bounds on its size than known before. Roughly speaking, the idea behind our arguments is to describe the genus of the division algebra locally in terms of the genus of its "unramified component" and the "ramified component". The genus of the unramified component  is directly related to the genus of the associated residue algebra which can be computed either directly or through an induction process. Once we have local information on the genus, we use some kind of  local-global information to compute the genus globally. \\
\indent We use the above techniques to give a short proof of the \textit{Stability theorem} (\cite[Theorem 3.5]{rapinchuk_genus_unramified}) for  purely transcendental extensions in \S \ref{sec:purely_trans}. We also give an explicit  description for the elements in the genus of division algebras over $k(x)$ where $k$ is a number field  and show that the bounds on the size of the genus thus obtained are tighter than known before. Next in \S \ref{sec:gcdvf}, we analyze the genus problem for tame algebras over complete discrete valued fields and show how it is related to the genus of residue algebras. In this section, we also consider a different notion of genus which we call as the \textit{splitting genus}, denoted by $gen^s_{spl}$. By definition $gen^s_{spl}([D])$ is the collection of  classes of division algebras sharing the same separable finite dimensional splitting fields with $D$\footnote{A similar notion without separability assumption is considered in \cite{gen_spl_krashen}}. Based on a decomposition lemma, we describe elements in $gen^s_{spl}([D])$ in terms of the genus of its ramified and unramified components.  Finally, in \S \ref{sec:qdvt}, we apply the above techniques for the special case of quaternion algebras  over a field $K$ with a set of discrete valuations $V$. For $v \in V$, let $K_v$ and $\overline{K_v}$ denote respectively the completion and residue field of $K$ with respect to $v$. Assume for every $v \in V$, that $\overline{K_v}$ satisfies the property that the genus of any quaternion algebra over $\overline{K_v}$ is trivial.  Then we show that  for a quaternion $Q$ over $K$ that is tame with respect to $V$, any element in $gen^s([Q])$  is Brauer equivalent to $Q \otimes_K Q'$ for some $[Q'] \in {}_2\Sha^{Br}(K, V)$. Here
\begin{align*}
   {}_2\Sha^{Br}(K, V)  := ker( {}_{2}Br(K) \rightarrow \prod_{v \in V}{}_2Br(K_v))
\end{align*}
is the $2$-torsion component of the kernel of the local-global map on the Brauer group with respect to $V$.  Hence we get that $|gen^s([Q])| \leq  |{}_2\Sha^{Br}(K, V)|$.   Comparing with  the bound $|gen^s([Q])| \leq |{}_2Br(K)_{V}|$ given in    \cite[Theorem 2.2]{rapinchuk_genus_unramified}, we see that we get a better upper bound on the genus  since  the residue maps factor through the local-global maps and therefore ${}_2\Sha^{Br}(K, V) \subseteq {}_2Br(K)_{V}$ (see (\ref{eqn:sha_unramified})). Continuing with our examples before,  for $K=\mathbb{Q}(\!(t)\!)$ with $V$ being the obvious $t$-adic valuation, we get  ${}_2\Sha^{Br}(K, V) =0$ trivially and  for $K=\mathbb{Q}(\!(t)\!)(C)$  with  $V$ the set of discrete valuations arising from codimension one points of some  regular proper model of $C$ over $k[[t]]$, ${}_2\Sha^{Br}(K, V)$  is also trivial by \cite[Theorem 4.2 and Theorem 4.3 (ii)]{parimala_local_global}. Hence $gen^s([Q])$ is trivial for any quaternion over $K$ for the above fields although ${}_2Br(K)_{V}$ is not even finite as shown before.  We give a number of examples of fields over which  the genus of any quaternion is trivial using the above techniques. These include higher local fields, function fields of curves over higher local fields (these are  semi-global fields and have been extensively studied in over the recent years)  and function fields of curves over real closed fields.\\
\indent In the special case when $K$ is the function field of a curve $C$ over a global field with a rational point, we show that the size of genus of any quaternion algebra over $K$ is bounded by size of the $2$-torsion subgroup of the Tate- Shafarevich group of the Jacobian, ${}_2\Sha(J_C)$. In particular, when $C =E$ is an elliptic curve we have $|gen^s([Q])| \leq |{}_2\Sha(E)|$ and therefore is trivial whenever ${}_2\Sha(E)$ is trivial. In  the cases where $ {}_2\Sha(E)$ is not trivial, we demonstrate  how the bounds on the genus can be improved by directly using the arithmetic properties of $E$. 

\section{Notation}
Although the term genus was originally coined for division algebras in \cite{rapinchuk_genus_unramified}, in this paper, we denote by $gen^s(\alpha)$ (resp. $gen^s_{spl}(\alpha)$) for $\alpha \in Br(K)$,  the Brauer classes whose underlying division algebras have   the same   separable maximal subfields (resp. same finite dimensional separable splitting fields) as the underlying division algebra of  $\alpha$. This is a well defined notion and makes it convenient to talk about the genus of a Brauer class.\\
\indent  For a division algebra $D$ over $K$, $[D]$ denotes its Brauer class in $Br(K)$. We write $D\sim E$ if $[D] = [E]$ in $Br(K)$. If $D$ is a valued division algebra, its residue algebra is denoted by $\overline{D}$. If  $K$ has a  discrete valuation $v$, $D_v$ denotes $D \otimes_K K_v$ and  $\overline{D_v}$ denotes the residue algebra of the underlying division algebra of  $D_v$.  If $K$ contains a primitive $n$-th root of unity $\omega$, the symbol algebra  generated  by $i,j$  over $K$ with the relations $i^n =a, j^n=b, ij = \omega ji$ is denoted by $(a,b)_{\omega}$. \\
\indent The $n$-torsion part of a group $G$ is denoted by ${}_nG$. For $g\in G$, $ord(g)$ denotes the order of $g$. The symbol $(m,n)$ denotes the greatest common divisor of $m$ and $n$. The separable closure of $K$ is denoted by $K^{sep}$. All the valuations considered in this paper are discrete.

\section{Main Results}
For a division algebra $D$ over a field  $k$, let $\tilde{D} : = D\otimes_k k(x)$. 
\begin{thm} [{Genus over purely transcendental extension of number fields, \S\ref{sec:purely_trans}}] Let $k$ be a number field containing a primitive $n$-th root of unity $\omega$. For any $[D]\in {}_nBr(k(x))$, write $[D] \simeq [\tilde C \otimes R_1 \otimes R_2 \otimes \cdots \otimes R_{r}]$ where $[C] \in {}_nBr(k)$, each $[R_i]$ is ramified at exactly one  point in $\mathbb{P}_0^1 -\infty$ and nowhere else  and $[R_i] \sim \otimes_j(f_{ij}, g_{ij})_{\omega}$ where for every $j$ either $f_{ij}$ or $g_{ij}$ is a non-constant monic polynomial in $k[x]$(this is possible by Lemma \ref{lem:decom}). Then
\begin{align*}
    gen^s([D]) \subseteq \{ [\tilde{C'} \otimes R_1^{i_1} \otimes R_2^{i_2} \otimes \cdots \otimes R_r^{i_r}] | [C'] \in gen^s([C]) \text{~and~} (i_l, ord([R_l])) =1 ~\forall 1\leq l\leq r \}
\end{align*}
In particular, 
\begin{align*}
    |gen^s([D])| \leq |gen^s([C])|\phi(n)^r
\end{align*}
where $\phi$ is the Euler's Totient function. 
\end{thm}

Now assume that $k$ is a complete discrete valued field with residue $\overline{k}$.  Let $SBr(k)$ and $IBr(k)$ denote respectively the subgroup of  inertially  split (tame) and inertial (unramified) classes in $Br(k)$(\S \ref{sec:cdvf}).  

\begin{thm}[{Theorem \ref{thm:sub_genus}}]
Let  $[D] \in SBr(k)$  be of  prime index $p$. Then 
\begin{align*}
gen^s([D])\begin{cases}
\subseteq  \{ [C] \in IBr(k) | [\overline{C}] \in gen^s([\overline{D})]\} \text{~~if $[D]$ is unramified~~}\\
 =\{[D^{\otimes i}]| 1 \leq i \leq p-1\} \text{~~else}
\end{cases}
\end{align*}
In particular, 
\begin{align*}
|gen^s([D])| \begin{cases}
\leq|gen^s(\overline{D})|\text{~~if $[D]$ is unramified~~}\\
= p-1\text{~~else}
\end{cases}
\end{align*}
If $\overline{k}$ is perfect, then $\subseteq$  and $\leq$  in the above  expressions are equalities. Moreover,  all the above statements also hold if $gen^s([D])$ is replaced with $gen^s_{spl}([D])$.
\end{thm}

As an easy consequence,  we get
\begin{cor} [{Corollary \ref{cor:sub_genus_up}}]
Let $\overline{k}$ satisfy the property that  the genus (resp. splitting genus) is trivial for any quaternion algebra over $\overline{k}$. Then the genus (resp. splitting genus) of any tame quaternion algebra over $k$ is trivial.
\end{cor}
For a division algebra $D$, let $e_D$ denote the ramification index of $D$.  Recall the definition of $gen^s_{spl}$ from \S\ref{sec:intro}.

\begin{thm} [{Genus decomposition for $gen^s_{spl}$, Theorem  \ref{thm:genus_cdvf}}]
Let   $I$ be unramified and $N$ be  NSR (\S\ref{sec:cdvf}) division algebras over $k$.  Then 
\begin{align*}
    gen^s_{spl}([I \otimes_k N]) \subseteq \{ [I' \otimes_k N']~ |~ [I']  \in gen^s([I]) \text{~~  and~~} [N'] \in gen^s([N])\}
\end{align*}
\end{thm}

\begin{cor}[{Corollary \ref{cor:splgenus}}]
Let  $[D] \in SBr(k)$. Write $D \sim I \otimes_k N$  where  $I$ is inertial and $N$ is NSR. Then any element in $gen^s_{spl}([D])$ is of the form $[I' \otimes_k N^{\otimes j}]$ for some  $(j,e_D) =1$ where $[I'] \in gen^s([I])$. The algebra $I'$ is the (unique) inertial lift of some division algebra  whose class lies in $gen^s([\overline{I}])$. In particular
\begin{align*}
|gen^s_{spl}([D])| \leq |gen^s([\overline{I}])|\cdot \phi(e_D)
\end{align*}
where $\phi$ denotes the Euler's Totient function. If moreover $\overline{k}$ is perfect, then  $I'$ above is the (unique) inertial lift of some  division algebra whose class lies in $gen^s_{spl}([\overline{I}])$ and
\begin{align*}
|gen^s_{spl}([D])| \leq |gen^s_{spl}([\overline{I}])|\cdot \phi(e_D)
\end{align*}
\end{cor}

Now let $K$ is an arbitrary field of $char \neq 2$ with a set of discrete valuations $V$.  We denote the completion of $K$  with respect to $v$ and its residue field respectively by $K_v$ and $\overline{K_v}$. Recall from \S\ref{sec:intro} that ${}_2\Sha^{Br}(K, V)$ denotes the $2$-torsion of the  kernel of the local-global map on the Brauer group with respect to $V$.

For a division algebra $D$ over $K$, we say that $D$ (or its class  in $Br(K)$) is tame  with respect to $V$ if $D_v:=D \otimes_K K_v$ is tame for every $v \in V$.  We denote that set of tame elements  of $Br(K)$ with respect to $V$ by $SBr(K,V)$.  
\begin{thm}(Theorem \ref{thm:main})
Let $[Q]\in SBr(K,V)$ be the class of a quaternion division algebra over $K$. Suppose for every $v\in V$, the genus of  any  quaternion division algebra  over $\overline{K_v}$ is trivial. Then
\begin{align*}
gen^s([Q]) \subseteq [Q]+ {}_2\Sha^{Br}(K, V)
\end{align*}
In particular,
\begin{align*}
    |gen^s([Q])| \leq |{}_2\Sha^{Br}(K, V)|
\end{align*}
\end{thm}

\begin{cor} [{Corollary \ref{cor:examples})}]
If $K$ is one of the following fields, the genus of any quaternion division algebra over $K$ is trivial
\begin{enumerate}
    \item  Higher local fields where the final residue field has characteristic $\neq 2$
    \item Iterated Laurent series $k(\!(t_1)\!)(\!(t_2)\!)\cdots(\!(t_n)\!)$ (resp. their finite extensions) where $char~k \neq 2$ and every quaternion division algebra over $k$ (resp. every finite extension of $k$) has trivial genus
    \item Function fields of  one variable over fields in (1)
    \item Function fields of one variable any real closed field 
    \item Function fields of  one variable  over fields in (2) where  for any curve $C$ over $k$ (where $k$ is as in (2)),  every quaternion algebra over every finite extension of $k$ and over $k(C)$ has trivial genus (for example, $k(\!(t_1)\!)(\!(t_2)\!)\cdots(\!(t_n)\!)(C)$  where one can take $k$ to be  any  real closed field by  (4))
\end{enumerate}
\end{cor}

For  an abelian variety $A$ over a global field $k$, let  $\Sha(A)$  denote the Tate-Shafarevich  group.
\begin{thm}[{\S \ref{sec:elliptic}}]
Let $C$ be a smooth projective geometrically integral curve over a global field $k$ with a rational point. Let $Q$ be a quaternion algebra  over $k(C)$ that is  tame with respect to every (dyadic) discrete valuation arising from a codimension one point of  a regular projective model of $C$.  Then 
\begin{align*}
    |gen^s([Q])| \leq |{}_2\Sha(J_C)|
\end{align*}
where $J_C$ is the Jacobian of $C$. In particular, when $C=E$ is an elliptic curve, we have
\begin{align*}
    |gen^s([Q])| \leq |{}_2\Sha(E)|
\end{align*}
In particular, we get $gen^s([Q])$ is trivial whenever ${}_2\Sha(E)$ is trivial.
\end{thm}

\section{Preliminaries}
\subsection{Division algebras  over complete discrete valued fields} \label{sec:cdvf}
Let  $D$ be a  finite dimensional division algebra over a complete discrete valued field $k$ with residue field $\overline{k}$. Since $k$ is complete, the valuation on $k$ extends uniquely to a valuation $v$ on $D$ (\cite[Corolary 2.2]{wadsworth}).   Associated to $v$, one can define the residue algebra (\cite[\S2]{wadsworth} 
\begin{align*}
    \overline{D} = V_D/M_D
\end{align*}
where 
\begin{align*}
V_D &= \{ a \in D^* | v (a) \geq 0\} \cup 0\\
M_D &= \{a \in D^* | v (a) > 0 \} \cup 0
\end{align*}
One can easily see that $\overline{D}$ is a division algebra over its center $Z(\overline{D})$ (which may or may not be equal to the residue field $\overline{k}$). Let $\Gamma_D$ and $\Gamma_k$ denote respectively the value group of $D$ and $k$. Note that, the valuation $v$ is discrete and we have  $\Gamma_k \subseteq \Gamma_D \subseteq \frac{1}{\sqrt{[D:k]}}\Gamma_k$ (\cite[Eqn. (2.7)]{wadsworth}). Therefore, $\Gamma_D/ \Gamma_k$ is cyclic and the number $e_{D} :=|\Gamma_D:\Gamma_k|$ is called the \textit{ramification index} of $D$. Moreover, we have the following \textit{fundamental equality} (\cite[Equation (2.10)]{wadsworth}, \cite[page 359]{morandi_defective})
\begin{align}\label{eqn:fund_equality}
    [D:k] = [\overline{D}:\overline{k}] [\Gamma_D: \Gamma_k]
\end{align}
We say that $D$ is \textit{unramified or inertial} if $Z(\overline{D}) = \overline{k}$  and $[\overline{D}:\overline{k}] = [D:k]$. An interesting fact similar to the case of fields is that, given any division algebra $\tilde{D}$ over $\overline{k}$, there is a unique unramified (also called inertial) division algebra $D$ over $k$ with the property that $\overline{D} \simeq \tilde{D}$. This is called the \textit{inertial lift} of $\tilde{D}$ over $k$ (\cite[Theorem 2.8(a)]{wadsworth_henselian}).
Note that if $D$ is unramified, the ramification index $e_D =1$ by the above equality. When $e_D$ is larger that $1$, $D$ is said to be \textit{ramified}.  One special case of ramified division algebras is that of \textit{nicely semiramified} (abbreviated as NSR) algebras. A division algebra $N$ over $k$ is said to be NSR  if it contains a maximal subfield that is unramified  as well as a  maximal subfield that  is totally ramified   $k$ (\cite[\S4]{wadsworth_henselian}, \cite[Theorem 2.4]{karim_nsr}).  In this case, $\overline{N}$ is a field and 
\begin{align*} 
    [\overline{N}:\overline{k}] = [\Gamma_N: \Gamma_k] = \sqrt{[N:k]}
\end{align*}
See \cite[\S 4]{wadsworth_henselian} for more details. 
\\
\indent We say that $D$ is \textit{ inertially split} if $D$ is split by an unramified (inertial) extension of $k$. By \cite[Lemma 6.2]{wadsworth_henselian} and \cite[Remark 3.2(a)]{tignol_wadsworth_ramified}, this is equivalent to $D$ being \textit{tame} as defined in  \cite[Definition \S6]{wadsworth_henselian}.  When $\overline{k}$ is perfect or if $char~\overline{k}$ is coprime to the degree of $D$, every $D$ over $k$ is tame (\cite[Theorem 1, Exercise 3, Chapter XII]{serre_local}).  Note that an unramified  division algebra $D$ is automatically tame since it contains the inertial lift of a separable maximal subfield  of $\overline{D}$ (\cite[Theorem 2.9]{wadsworth_henselian}).   Following \cite{wadsworth} we define
\begin{align*}
IBr(k) &= \{[D] | D \text{~is ~a~ finite ~dimensional ~ unramified~ division~ algebra~ over ~} k  \}\\
    SBr(k) &= \{[D] | D \text{~is ~a~ finite ~dimensional ~ tame~ division~ algebra~ over ~} k  \}
\end{align*}
We observe the following:
\begin{prop}\label{prop:nsr}
    Let $D$ be finite dimensional  tame division algebra over $k$ such that $\overline{D}$ is a field. Then $D$ is cyclic. In particular, any NSR division algebra over $k$ is cyclic.
\end{prop}
\begin{proof}
    By \cite[Lemma 5.1]{wadsworth_henselian}, $\overline{D}/\overline{k}$ is a cyclic field extension and  \begin{align*}
    [\overline{D}:\overline{k}] = [\Gamma_D: \Gamma_k] = \sqrt{[D:k]}
\end{align*} 
by (\ref{eqn:fund_equality}). Now $D$ contains the inertial lift $L$ of $\overline{D}$ by \cite[Theorem 2.9]{wadsworth_henselian}. Note that  $L$ is a maximal subfield and is  cyclic(\cite[Chapter III, \S5, Theorem 3]{serre_local})).  Therefore $D$ is cyclic.  The last statement follows from the previous statement and fact that any NSR division algebras is tame (since it contains an unramified maximal subfield) and  the   residue  algebra is a field. 
\end{proof}

\begin{thm}[{\cite[Lemma 5.14, Theorem 5.15]{wadsworth_henselian}}]
Given any   $[D] \in SBr(k)$, there is a (non-canonical) decomposition in $Br(k)$ given by
\begin{align*}
    D \sim I \otimes_k N
\end{align*}
where $I$ is unramified and $N$ is NSR . Moreover $Z(\overline{D}) = \overline{N}, \overline{D} \simeq \overline{I}\otimes_{\overline{k}}, \overline{N}, e_D = e_N$
\end{thm}

\indent Since non-tame division algebras pose certain difficulties that we are not able to handle at the moment, we will consider only  tame algebras in this paper.

\subsection{The unramified Brauer group}\label{sec:unramified}
Let $K$ be a field equipped with a  discrete valuation $v$. Let $K_v$, $\mathcal{O}_v$ and $\overline{K_v}$ denote respectively the completion of $K$ at $v$, its valuation ring  and the residue field.  For every $n$ such that $n$ is coprime to  $char~\overline{K_v} $ or when $\overline{K_v}$ is perfect,  one can define the \textit{residue homomorphism} (\cite[Chapter 10]{saltman_division})
\begin{align*}
    \rho_v^n: {}_nBr(K_v) &\rightarrow H^1(\overline{K_v}, \mathbb{Z}/n\mathbb{Z}) \\
    [D] &\rightarrow \chi_D
    \end{align*}
We recall that  the image of $\chi_D$ is $\Gamma_D/\Gamma_{K_v}$ (\cite[Theorem 5.6 (b)]{wadsworth_henselian}) and the fixed field of $ker(\chi_D)$ is the center of $\overline{D}$ (\cite[Theorem 3.5]{wadsworth}).  The kernel of  $\rho_v^n$ can be identified with ${}_nBr(\mathcal{O}_v)$ and is  the $n$-torsion component of the \textit{unramified Brauer group of $K$ at $v$}.  A non-trivial class $[D] \in Br(K_v)$ is unramified at $v$ if and only if $D$ is unramified  in the sense of \S \ref{sec:cdvf} (\cite[Equation (3.9)]{wadsworth}). Therefore we have ${}_nIBr(K_v) = ker (\rho_n^v)$ and there is isomorphism(\cite[Equation (3.7)]{wadsworth})

\begin{align}\label{eqn:unramified_residue}
IBr(K_v) &\simeq Br(\overline{K_v})\\
D &\mapsto \overline{D} \nonumber
\end{align}

Now let $K$ be equipped with a set of discrete valuations $V$.  Assuming $char~\overline{K_v}$ is coprime to $n$ or $\overline{K_v}$ is perfect for every $v \in V$, we have the residue map relative to $V$ given by
\begin{align*}
    \rho_V^n: {}_nBr(K) \rightarrow \prod_{v\in V} {}_nBr(K_v) \rightarrow \prod_{v \in V} H^1(\overline{K_v}, \mathbb{Z}/n\mathbb{Z}) 
\end{align*}
where the first map is the local-global map on the  Brauer group with respect to $V$ whose kernel is denoted by ${}_n\Sha^{Br}(K,V)$. We say that the \textit{local-global principle holds on ${}_nBr(K)$ with respect to $V$} if  ${}_n\Sha^{Br}(K,V)$ is trivial.  The $n$-torsion of the \textit{unramified Brauer group of $K$ with respect to $V$} is defined to be
\begin{align*}
   {}_nBr(K)_V := Ker~\rho_V^n
\end{align*}
We note that 
\begin{align}\label{eqn:sha_unramified}
  {}_n\Sha^{Br}(K,V) \subseteq {}_nBr(K)_V
\end{align}

Let $X$ be a regular integral noetherian scheme  with function field $k(X)$ and let $V$ be  the valuations corresponding to the codimension 1 points  $X^{(1)}$  of $X$. Let $k(x)$ denote the residue field at $x\in X^{(1)}$.   Then we have an exact sequence  (\cite[Theorem 6.8.3]{poonen_rational})
\begin{align*}
    0 \rightarrow Br(X) \rightarrow Br(k(X)) \xrightarrow{\rho_V} \bigoplus_{x \in X^{(1)}} H^1( k(x), \mathbb{Q}/\mathbb{Z})
\end{align*}
where we exclude the $p$-primary components in all of the above groups if $k(x)$ is imperfect of characteristic $p$ for some $x$ (\cite[Theorem 6.8.3]{poonen_rational}, \cite[Remark 6.4]{purity_brauer}).

\subsection{Higher local fields}\label{sec:higher_local}
We will quickly state  a few facts  on higher local fields that we use in this paper. We will use the notations from (\cite[Definition 2.1]{morrow_higher_local}).\\
\indent Given a field $F$,  the \textit{complete discrete valuation dimension }  of $F$, denoted $cdv.dim(F)$ is defined to be 
\begin{align*}
    cdv.dim(F) := \begin{cases}
    0 \text{~~if $F$ is not compete discrete valued field}\\
    cdv.dim(\overline{F}) + 1 \text{~~else}
    \end{cases}
\end{align*}
where $\overline{F}$ is the residue field. \\
\indent Set $F^{(0)} := F$. For every $i$, if $F^{(i)}$ is a  complete discrete valued field, one can define $F^{(i+1)}$ to be the residue field of $F^{(i)}$. A field $F$ is said to be an \textit{$n$-dimensional local field} for some
$n \geq 0$ if  $
(F) = n$ and the final residue field $F^{(n)}$ is finite.\\
\indent By \cite[Proposition 2.15]{morrow_higher_local}, every finite extension of an $n$-dimensional local field is an $n$-dimensional local field.  The classification of $n$-dimensional local fields can be found in \cite[\S 1.1]{higher_local}, \cite[\S 2.1]{morrow_higher_local} (although we will not be needing it in this paper).

\subsection{An important lemma} We will repeatedly use a lemma at many places in genus computations. For a division algebra $D$ over $K$ equipped  with a discrete valuation $v$, let $D_v$ denote $D \otimes_K K_v$ and  $\overline{D_v}$ denote the residue algebra of the underlying division algebra of  $D_v$.

\begin{lemma}\label{lem:imp}
Let $D$ be a division algebra over a field $K$ with discrete valuation $v$. Assume that $D$ is tame with respect to  $v$. If  $[D'] \in gen^s([D])$, then $[\overline{D'_v}] \in gen^s([\overline{D_v}])$.
\end{lemma}
\begin{proof}
By \cite[Corollary 2.4]{RR_genus}, the underlying division algebras of $D_v$ and $D'_v$ have the same maximal separable subfields. We may assume that $D_v$ (and hence $D'_v$) are division algebras.  Since $D_v$ (and hence $D_v'$) is tame, the center $Z(\overline{D_v})$ is separable over $\overline{K_v}$ (\cite[Theorem 3.4(iii)]{wadsworth}). Moreover, by \cite[Lemma 2.3 and Remark 2.6]{rapinchuk_genus_unramified},  $Z(\overline{D_v}) \simeq Z(\overline{D'_v})$  over $\overline{K_v}$.  Let $\tilde{L} \subset \overline{D_v}$ be a   separable maximal subfield over $Z(\overline{D_v})$ (hence separable over $\overline{K_v}$) and let $L$ be the inertial lift of $\tilde{L}$ over $K_v$. Then by \cite[Theorem 2.9]{wadsworth_henselian}, we conclude that $L \subset D_v$ is a     separable subfield. By the fundamental equality (\ref{eqn:fund_equality}) and \cite[Lemma 5.1(iii)]{wadsworth_henselian}, $L$ is a separable maximal subfield  of $D_v$. By  hypothesis, $L$ is also a  separable maximal subfield of $D'_v$. Hence, $\tilde{L} \subset \overline{D'_v}$ is a    separable maximal subfield. Since the above argument   is symmetric with respect to $D_v$ and $D'_v$, we conclude that $ [\overline{D'_v}] \in gen^s([\overline{D_v}])$.
\end{proof}

\section{Genus over purely transcendental extensions}\label{sec:purely_trans}
Let $k$ be a field with $n$ invertible. Recall the Faddeev's split exact sequence (\cite[Corollary 6.9.3]{gille_sam}):
\begin{align}\label{eqn:faddeev}
0 \rightarrow {}_nBr(k) \rightarrow {}_n Br(k(x)) \xrightarrow{\rho:= \oplus \rho_P} \oplus_{P \in \mathbb{P}_0^1 -\infty} H^1(k(P), \mathbb{Z}/n\mathbb{Z}) \rightarrow 0
\end{align}
where the third arrow coincides with the residue map in Galois cohomology with respect to the discrete valuation corresponding to the points $P \in \mathbb{P}_0^1 -\infty$.  We say that an element in $Br(k(x))$ is ramified or unramified at $P$, if it is so for the discrete valuation corresponding to $P$. 
\\
\indent The  exact sequence (\ref{eqn:faddeev}) splits. Let 
\begin{align*}
\theta: \oplus_{P \in \mathbb{P}_0^1 -\infty} H^1(k(P), \mathbb{Z}/n\mathbb{Z}) \rightarrow {}_nBr(k(x))
\end{align*}
be a splitting.  So any class $[D] \in {}_nBr(k)$ can be written as 
\begin{align}\label{eqn:decomp_trans}
[D] \sim [\tilde C \otimes_{k(x)} R]
\end{align}
where $[\tilde C] \in {}_nBr(k(x))$ is the restriction of a unique $[C]  \in {}_nBr(k)$ and $[R]$ is an element in the image under $\theta$ . Write $[R] = [R_1 \otimes R_2 \otimes \cdots \otimes R_{r}]$ where $r$ is the size of the ramification locus of $D$ i.e,   number of points $P\in \mathbb{P}_0^1 -\infty$ at which the residue map is non-zero and $[R_i]$ are the image  of the ramification components under  $\theta$. \\
\indent For any  division algebra $D$ over $k(x)$ and $P \in \mathbb{P}_0^1 - \infty$, let  $D_P:= D \otimes_{k(x)} k(x)_P$ where $k(x)_P$ is the completion of $k(x)$ with respect to the discrete valuation corresponding to $P$ and let $\overline{D_P}$ denote the residue algebra of the division algebra in the class $[D_P]$. Similar notations, $\alpha_P$ and $\overline{\alpha_P}$  are used for any $\alpha \in Br(k(x))$.
\\
\indent The decomposition of $D$ given in (\ref{eqn:decomp_trans}) into unramified and ramified components easily yields  a short proof of  the \textit{Stability Theorem}.
\begin{thm}[{Stability Theorem, \cite[Theorem 3.5]{rapinchuk_genus_unramified}}] 
 Let $k$ be a field  of characteristic $\neq 2$ with the property that the genus of any element in ${}_2Br(k)$ is trivial. Then the same property holds for ${}_2Br(k(x))$.
\end{thm}
\begin{proof}
Let $[D],[E] \in {}_2Br(k(x))$ be such that $[E] \in gen^s([D])$. Write 
\begin{align*}
[D] \sim [\tilde{C} \otimes R_1 \otimes R_2 \otimes \cdots \otimes R_{r}]
\end{align*}
where $[\tilde{C}] \in {}_2Br(k(x))$ is the restriction of a unique $[C]  \in {}_2Br(k)$ and each $[R_i]$ is ramified at exactly $P_i \in \mathbb{P}_0^1 -\infty$ and nowhere else.  By \cite[Lemma 2.5]{rapinchuk_genus_unramified}, $\rho([E]) = \rho([D])$ . Therefore,
\begin{align*}
     [E] \sim [\tilde{C'} \otimes R_1 \otimes R_2 \otimes \cdots \otimes R_r]  
 \end{align*}
 where $[\tilde{C'}] \in {}_2Br(k(x))$ is the restriction of a unique $[C']  \in {}_nBr(k)$.
 We need to show $[C'] \sim [C]$. Let $P = \mathbb{P}_0^1(k) -\infty$ be a $k$-rational point where each $[R_i]$ is unramified. Note that $D$ and $E$ are unramified at $P$ and  recall that  the map (\ref{eqn:unramified_residue}) is a homomorphism. Therefore residue algebras at $P$ are given by
\begin{align*}
    [\overline{E_P}] \sim [C' \otimes \overline{R_{1P}} \otimes \overline{R_{1P}}  \otimes \cdots \otimes \overline{R_{1P}} ] \\
     [\overline{D_P}] \sim [C \otimes \overline{R_{1P}}  \otimes \overline{R_{1P}}  \otimes \cdots \otimes \overline{R_{1P}} ] 
\end{align*}
in ${}_2Br(k)$. Here $[\overline{\tilde{C}_{P}}] = [C]$ because the composite
\begin{align*}
    Br(k) \rightarrow Br(k(x)) \rightarrow Br(k(x)_P) \rightarrow Br(\overline{k(x)_P})
\end{align*}
is identity.\\
\indent Since $[E] \in gen^s([D])$, by Lemma \ref{lem:imp}, we get  
\begin{align*}
[\overline{E_P}] \in gen^s([\overline{D_P}])
\end{align*}
By hypothesis the genus of any element in ${}_2Br(k)$ is trivial. Therefore, $[\overline{E_P}]\sim [\overline{D_P}]$  and we conclude  that $[C'] \sim [C]$.
\end{proof}

It is tempting to know if the above theorem can be generalized to compute a formula for the genus of elements in ${}_nBr(k(x))$ for $n\geq 3$. We will now derive a formula for genus of arbitrary elements in  ${}_nBr(k(x))$ for $n\geq 3$ when $k$ is a number field containing primitive $n$-th root of unity and show that resulting the bound on the size of the genus is sharp. The idea is to relate the genus of a  given Brauer class to the genus of its unramified component and the ramification data.
\\
\indent Suppose $k$ is  any  field containing a primitive $n$-th root of unity $\omega$. By Merkurjev-Suslin theorem (originally due to S. Bloch for $k(x)$), the $n$-torsion of the Brauer group of $k(x)$ is generated by symbol algebras, so each element in $Br(k(x))$ is Brauer equivalent to $ \otimes_j(f_j, g_j)_{\omega}$ where $f_j, g_j \in k(x)$. By manipulation of symbol algebras (\cite[Chapter VII]{albert}), we can assume that $f_j, g_j \in k[x]$ and that
\begin{align*}
    (f_j, g_j)_{\omega} \sim (\otimes_l(a_l,b_l)_{\omega} )\otimes (\otimes_m (f_m, g_m)_{\omega})
\end{align*}
where $a_l,b_l \in k^*$ and for each $m$, either $f_m$or $g_m$ is a non-constant monic polynomial. This observation together with the split exact sequence (\ref{eqn:faddeev}) yields:

\begin{lemma} \label{lem:decom}
For any $[D] \in {}_nBr(k(x))$, we have  
 $D \sim \tilde{C} \otimes R_1 \otimes R_2 \otimes \cdots \otimes R_{r}$
where $[\tilde{C}] \in {}_nBr(k(x))$ is the restriction of a unique $[C]  \in {}_nBr(k)$ and each $[R_i]$ is ramified at exactly $P_i \in \mathbb{P}_0^1 -\infty$ and nowhere else. Moreover  $[R_i] \sim \otimes_j(f_{ij}, g_{ij})_{\omega}$ where for every $j$ either $f_{ij}$or $g_{ij}$ is a non-constant monic polynomial in $k[x]$.
\end{lemma}
For the rest of this section, let $k$ denote  a number field.
\begin{thm} \label{thm:genus_trans_sub} Let $n\geq 3$ and  $k$ be a number field containing a primitive $n$-th root of unity. Let $D \sim \tilde C \otimes R_1 \otimes R_2 \otimes \cdots \otimes R_{r}$ as in Lemma \ref{lem:decom}. Then
\begin{align*}
    gen^s([D]) \subseteq \{ [\tilde{C'} \otimes R_1^{i_1} \otimes R_2^{i_2} \otimes \cdots \otimes R_r^{i_r}] | [C'] \in gen^s([C]) \text{~and~} (i_l, ord([R_l])) =1 ~\forall 1\leq l\leq r \}
\end{align*}
In particular,
\begin{align*}
    |gen^s([D])| \leq |gen^s(C)|\phi(n)^r
\end{align*}
where $\phi(n)$ is the Euler's Totient function.
\end{thm}

 \begin{remark}\normalfont
 One can view $\tilde{C}$ and $R_1 \otimes R_2 \otimes \cdots \otimes R_{r}$  in Lemma \ref{lem:decom} respectively as unramified and ramified components of $D$ (note that this decomposition is not unique). Then Theorem \ref{thm:genus_trans_sub} expresses the genus of $D$ in terms of the genus of its unramified component  and  the ramification data. Note that the genus of the unramified component $\tilde{C}$ is  expressed in terms of genus of the residue algebra $C$ with respect to the valuation corresponding to  any  rational point. 
 \end{remark}
 \indent In order to prove the theorem, we need another lemma.   Note that for a class $[D] \in Br(k(x))$ that is unramified at  $P \in \mathbb{P}_0^1(k) -\infty$, the residue class $[\overline{D_P}]$ is in $Br(k)$. For a place  $v$  of $k$, let $[\overline{D_P}]_v$  denote the  completion of $[\overline{D_P}]$  at the place $v$ of $k$.  Let $PF(k)$ denote the set of finite places of $k$.
 
 \begin{lemma}\label{lem:double_val}
 In the notation of Theorem \ref{thm:genus_trans_sub}, given any $v\in PF(k)$, there exists a $k$-rational point $P_v \in \mathbb{P}_0^1(k) -\infty$ such that for every $i$, $[R_i]$ is unramified at $P_v$ and $[\overline{(R_{i})_{P_v}}]_v$ is trivial in  $Br(k_v)$.
 \end{lemma}
 \begin{proof}
  By Lemma \ref{lem:decom},   $[R_i] \sim \otimes_j (f_{ij}, g_{ij})_{\omega}$ where  for every $j$ either $f_{ij}$ or $g_{ij}$ is a non-constant monic polynomial. Without loss of generality assume that $g_{ij}$ is a non-constant monic polynomial  for every $i$ and $j$. Outside a finite set $S$ of points  in $\mathbb{P}_0^1(k) -\infty$, the extensions $k(x)(\sqrt[n]{f_{ij}})/k(x)$ and $k(x)(\sqrt[n]{g_{ij}})/k(x)$ are unramified for every $i,j$. Let $\pi_v\in k$ be a uniformizer in $\mathcal{O}_v$.   Let $m\in\mathbb{Z}$ be smallest valuation (with respect to $v$) of the set of coefficients in $g_{ij}$ for all $i,j$.   Let $N$ be  a non-negative integer  such that $n|N$ and  $N\geq -m +1 $. Set $P_N := \frac{1}{\pi_v^N} \in \mathbb{P}_0^1(k) -\infty - S$.  Then note that 
  \begin{align*}
       g_{ij}(P_N) = \pi^{-deg(g_{ij})N} \cdot b_{ij}
  \end{align*}
  where  $b_{ij}\in (1+\pi_v^{N+m}\mathcal{O}_v)$ . Therefore, 
  \begin{align*}
     (\overline{(f_{ij}, g_{ij})_{\omega})_{P_N}}\sim (a_{ij}, b_{ij})_{\omega}
  \end{align*}
for some $a_{ij}\in k^*$.  Since $k_v$ is a local field, there are only finitely many  extensions of degree at most $n$ (\cite[Chapter II, \S 5, Proposition 14]{lang_number_theory}).  Let $M$ be the supremum of the conductors  of these extensions. So by choosing $N\geq M-m$ together with the above conditions on $N$, we deduce  that $b_{ij} \in Norm({k_v(\sqrt[n]{a_{ij}})/k_v)})$ and hence $(a_{ij}, b_{ij})_{\omega} \otimes_k k_v$ are split for all $i,j$. Setting $P_v = P_N$ proves the lemma.
 \end{proof}
 We are now ready to prove Theorem \ref{thm:genus_trans_sub}.\\
 \noindent \textit{Proof of Theorem \ref{thm:genus_trans_sub}:}
 Let $[E] \in gen^s([D])$. We now make the following observation. By \cite[Lemma 2.5] {rapinchuk_genus_unramified}, $[E]$ is unramified  at a valuation if and only if $[D]$ is. Moreover, for any point $P$,  $\rho_P([E]) =\rho_P([D^{\otimes {i_P}}])$  for some $i_P$  coprime to $ord([D])$.  Recall that  each $[R_i]$ is ramified  exactly at $P_i$ and nowhere else, so by Faddeev's exact sequence (\ref{eqn:faddeev}),  
 \begin{align*}
 \rho([E]) =  \rho( [R_1^{i_1} \otimes R_2^{i_2} \otimes \cdots \otimes R_r^{i_r}]) 
 \end{align*}

Therefore we get 
 \begin{align*}
     [E] \sim [\tilde{C'} \otimes R_1^{i_1} \otimes R_2^{i_2} \otimes \cdots \otimes R_r^{i_r}]  
 \end{align*}
 where $(i_l, ord(\rho[R_l]) =1) ~\forall  1\leq l\leq r$ and $C' \in Br(k)$.  So to prove the theorem  it remains show that $[C']\in gen^s([C])$. \\
 
Let $v\in PF(k)$. By Lemma \ref{lem:double_val}, there exists a $k$-rational point $P_v \in \mathbb{P}_0^1(k) -\infty$ such that for every $i$, $[R_i]$ is unramified at $P_v$ and $[\overline{R_{iP_v}}]_v$ is split over $k_v$. Therefore in ${}_nBr(k_v)$, we have 
 \begin{align*}
     [\overline{D_{P_v}}]_v \sim [\overline{\tilde{C}_{P_v}}]_v = [C]_v\\
     [\overline{E_{P_v}}]_v \sim [\overline{\tilde{C'}_{P_v}}]_v =[C']_v
 \end{align*}

By hypothesis, $[E] \in gen^s([D])$. Therefore by Lemma \ref{lem:imp} and \cite[Corollary 2.4]{RR_genus}, we conclude that $[C']_v \in gen^s([C]_v)$ for every $v\in PF(k)$. Since ${}_nBr(k_v)$ is cyclic and period equals index for Brauer classes over local fields, we deduce that for every $v\in PF(k)$, we have  $[C']_v \sim [C^{\otimes j_v}]_v$ for some $j_v$ such that $(j_v,n)=1$. Note that  since $k$ contains $n$-th root of unity, and $n\geq 3$, $k$ has no real embeddings. Now by \cite[\S 18.4, Corollary b]{pierce_assoc}, it is easy to see that $[C'] \in gen^s([C])$. 
This completes the proof. \\

\indent Theorem \ref{thm:genus_trans_sub} can be viewed as a  generalization of   "genus stability" for higher dimensional division algebras which was queried in \cite[page 284]{RR_genus}.

 \begin{remark}\normalfont
The behaviour of genus of  division algebras of degree $\geq 3$ is studied in  \cite[Theorem 3.3]{rapinchuk_genus_unramified}. However, we  note that their results do not yield any finite upper bound for the genus over a number field $k$ since the size of $ gen^s(\Delta) \cap {}_nBr(k)$ is not bounded for $[\Delta] \in Br(k)$. See (\cite[\S 1]{rapinchuk_russian}). \\
\indent  Yet another bound  in \cite[Corollary 8.4]{size_genus} is given by 
\begin{align*}
    gen^s([D]) \leq \phi(n)^r \cdot n^{|S|}
\end{align*}
where $r$ is the number of ramification places of $D$ in $V$ (where $V$   and $S$ are given in \cite[\S8.2]{size_genus}). One can easily check that  our bound given in Theorem \ref{thm:genus_trans_sub} is better than the bound above. 
 \end{remark}
 
 \begin{remark} \normalfont
 The bound given in Theorem \ref{thm:genus_trans_sub} is also sharp. Let $k:=\mathbb{Q}(\omega)$ where $\omega$ is a primitive $n$-th root of unity for $n\geq 3$. Let $f\in k[x]$ be a  non-constant monic polynomial  and $a \in k^*$ such that  $D = (f,a)_{\omega}$,  is a division algebra ramified at exactly one point in $\mathbb{P}^1_0 - \infty$. Then Theorem \ref{thm:genus_trans_sub} yields:
 \begin{align*}
     gen^s([D]) \subseteq \{[D^{\otimes i}]| (i,n) =1\}
 \end{align*}
But the above inclusion is equality since  $[D^{\otimes i}]$ and $[D]$ generate the same subgroup of the Brauer group whenever $(i,n) =1$.
\end{remark}

   The proof of Theorem \ref{thm:genus_trans_sub} makes use of results from number fields and local fields. So it cannot be directly used to prove a similar statement for purely transcendental extensions of  arbitrary fields. There are no results so far in the literature to compute the genus of division algebras over such fields of degree $\geq 3$ in terms of its ramified and unramified components. It would be interesting to know if this is true in general.
   
\begin{question}\normalfont
 Does Theorem \ref{thm:genus_trans_sub} hold if $k$ is replaced by an arbitrary field?
  \end{question}

\section{Genus  over complete discrete valued fields}\label{sec:gcdvf}
In this section, we  first give a formula for the genus of division algebras of prime degree over $k$. We then describe the splitting genus of division algebras of arbitrary degree over a complete discrete valued field in terms of the genus of its ramified and unramified components. \\
\indent Throughout this section $k$ denotes a complete discrete valued field with residue field $\overline{k}$.  Recall from \S\ref{sec:cdvf} that  $IBr(k)$ and $SBr(k)$ denote the subgroup of inertial and tame classes in $Br(k)$.  For a division algebra $D$ over $k$, let $\overline{D}$ denote the residue division algebra. If  $\alpha \in Br(k)$ denotes the class of $D$, we set  $e_{\alpha} := e_D = |\Gamma_D: \Gamma_k|$. 

\subsection{The genus of division algebras of prime degree}
First we study the behaviour of Brauer classes under totally ramified extensions of $k$.
\begin{lemma}\label{lem:totram}
Suppose  $D$ is an unramified division algebra over $k$ that is split by a totally ramified extension, then $D\cong k$.
\end{lemma}
\begin{proof}
This is basically  \cite[Theorem 1]{morandi_henselization} together with the fact that discrete valuations are defectless (\cite[page 359]{morandi_defective}).
\end{proof}

\begin{lemma}\label{lem:exponent}
Let $N$ be a tame division algebra  split by a totally ramified extension\footnote{This is equivalent to saying that $N$ is NSR by \cite[Theorem 2.4]{karim_nsr})}. Then $ord([N]) = e_{[N]}$.
\end{lemma}
\begin{proof}
Since $e_{[N]} \vert ord([N])$ (\cite[Corollary 6.10]{wadsworth_henselian}), it suffices to show that $[N^{\otimes e_{[N]}}]$ is trivial. Note that $[N^{\otimes e_{[N]}}]$ is unramified by (\cite[Theorem 5.6 (b)]{wadsworth_henselian}).  By hypothesis, $[N^{\otimes e_{[N]}}]$ is split by a totally ramified extension of $k$ and therefore is trivial by Lemma \ref{lem:totram}. 
\end{proof}

\begin{prop} \label{prop:genusramified}
Let $N$ be a tame division algebra  split by a totally ramified extension. Then
\begin{align*}
    gen^s_{spl}([N]) =  gen^s([N]) = \{ [N^{\otimes i}] | (i, e_{[N]}) = 1\}
    \end{align*}
\end{prop}
\begin{proof}
By Lemma \ref{lem:exponent}, if $ (i, e_{[N]}) = 1  $, $[N^{\otimes i}]$ and $[N]$ have same splitting fields (and hence same subfields) since they generate the same subgroup of the Brauer group.     Moreover,  $gen^s_{spl}([N]) \subseteq gen^s([N])$. So it suffices to show that if $ [M] \in gen^s([N])$  then $[M] = [N^{\otimes i}]$ for some $i$ such that $ (i, e_{[N]}) = 1$. Let $M \in gen^s([N])$. Then by \cite[Lemma 2.5]{rapinchuk_genus_unramified}, we have $[M\otimes_k N^{(\otimes i)op}]$ is unramified for some $i$ where $ (i, e_{[N]}) = 1$.  By \cite[Theorem 2.4]{karim_nsr}, $N$ is a NSR division algebra and therefore, $N$ is contains a totally ramified separable maximal subfield $L$ by Appendix \ref{appendix:totram}.  Since $[M] \in gen^s([N]) $, $[M\otimes_k N^{(\otimes i)op}]$ is split by $L$. Now we use Lemma \ref{lem:totram} to conclude that $[M] = [N^{\otimes i}]$.
\end{proof}

\begin{lemma} \label{lem:imp2}
    Suppose $D$ is  a tame division algebra over $k$. If $[D'] \in gen^s([D])$ (resp.$[D'] \in gen^s_{spl}([D])$), then $[\overline{D}] \in gen^s([\overline{D'}])$ (resp. $[\overline{D}] \in gen^s_{spl}([\overline{D'}])$).
\end{lemma}
\begin{proof}
    The fact that  $[D'] \in gen^s([D])$ implies  $[\overline{D}] \in gen^s([\overline{D}])$ follows from Lemma \ref{lem:imp}. Now suppose $[D'] \in gen^s_{spl}([D])$.  First note that $Z(\overline{D}) = Z(\overline {D'})$ (\cite[Lemma 2.3 and Remark 2.6]{rapinchuk_genus_unramified}).  Let $\overline{L} \subset \overline{D}$ be a separable finite dimensional splitting field over $Z(\overline{D})$ with inertial lift  $L$. Then  by \cite[Corollary 3.5]{wadsworth_henselian}, $L$ splits $D$ and hence also splits $D'$ by hypothesis. Therefore again by \cite[Corollary 3.5]{wadsworth_henselian}  $\overline{L}$ splits $\overline{D'}$. Since the argument is symmetric with respect ot $D$ and $D'$ we conclude that $[\overline{D}] \in gen^s_{spl}([\overline{D'}])$. 
\end{proof}

\begin{prop}\label{prop:inertial_residue}
Let $I$ be an unramified division algebra over $k$. Then the  isomorphism in (\ref{eqn:unramified_residue})
\begin{align*}
    \phi: IBr(k) &\rightarrow Br(\overline{k}) \\
     [D] &\mapsto [\overline{D}]
\end{align*}

induces an injective maps of sets
\begin{align*}
   gen^s([I]) &\rightarrow gen^s([\overline{I}]) \\
    gen^s_{spl}([I]) &\rightarrow gen^s_{spl}([\overline{I}])
   \end{align*}
 If moreover, $\overline{k}$ is perfect, then the above maps are bijective
\end{prop}

\begin{proof}
The first claim follows from  Lemma \ref{lem:imp2}. Now assume that $\overline{k}$ is perfect or $char ~\overline{k}$ is coprime to the degree of $I$. It remains to show that if $[\overline{J}] \in gen^s([\overline{I}])$ (resp. $[\overline{J}] \in gen^s_{spl}([\overline{I}])$), then $[J] \in gen^s([I])$ (resp. $[J] \in gen^s_{spl}([I])$) where $J$ is the unique inertial lift of $\overline{J}$ (\S\ref{sec:cdvf}). \\
\indent Let $L \subset I$ be a    separable subfield. Clearly, by the assumption on $\overline{k}$, $\overline{L} \subseteq \overline{I}$ is a   separable subfield and therefore is also a    separable subfield of $\overline{J}$. By \cite[Theorem 2.9]{wadsworth_henselian}, we  conclude that $L \subset J$ is a separable subfield. Since the above arguments are symmetric with respect to $I$ and $J$, we have $J \in gen^s([I])$. \\
\indent The proof of $[J] \in gen^s_{spl}([I])$ follows by similar argument as above by replacing the phrase "separable subfield" with "separable finite dimensional splitting field" together with   Lemma \ref{lem:maxunramified} and  \cite[Corollary 3.5]{wadsworth_henselian}.  We leave the details to the reader.
\end{proof}

We will now describe the genus of division algebras of prime degree in terms of its residue algebra and ramification.
\begin{thm} \label{thm:sub_genus}
Let  $[D] \in SBr(k)$  be of prime index $p$. Then 
\begin{align*}
gen^s([D])\begin{cases}
\subseteq  \{ [C] \in IBr(k) | [\overline{C}] \in gen^s([\overline{D})]\} \text{~~if $[D]$ is unramified~~}\\
 =\{[D^{\otimes i}]| 1 \leq i \leq p-1\} \text{~~else}
\end{cases}
\end{align*}
In particular, 
\begin{align*}
|gen^s([D])| \begin{cases}
\leq|gen^s(\overline{D})|\text{~~if $[D]$ is unramified~~}\\
= p-1\text{~~else}
\end{cases}
\end{align*}
If $\overline{k}$ is perfect, then $\subseteq$  and $\leq$  in the above  expressions are equalities. Moreover,  all the above statements also hold if $gen^s([D])$ is replaced with $gen^s_{spl}([D])$.
\end{thm}

\begin{proof}
By the fundamental equality (\ref{eqn:fund_equality}), we have that $e_D=[\Gamma_D: \Gamma_k]$ is either $1$ or $p$ since there are no tame  and totally ramified division algebras over complete discrete valued fields (\cite[Remark 3.2(a)]{tignol_wadsworth_ramified}). If $e_{[D]}= 1$, $[D]$ is unramified. Therefore, $gen^s([D]) \subseteq \{ C | \overline{C} \in gen^s(\overline{D})\}$  and $gen^s_{spl}([D]) \subseteq \{ C | \overline{C} \in gen^s_{spl}(\overline{D})\}$)  (with $\subseteq$ replaced with $=$ if $\overline{k}$ is perfect) by Proposition \ref{prop:inertial_residue}. If $e_{[D]} = p$, then $D$ contains an element $a$ such that  whose valuation generates $\Gamma_D/ \Gamma_k$ implying that $D$ contains a  totally ramified maximal subfield. Therefore by Proposition \ref{prop:genusramified},  $ gen^s_{spl}([D]) = gen^s([D]) =  \{[D^{\otimes i}]| 1 \leq i \leq p-1\}$.
\end{proof}

\noindent The case $p=2$ yields:
\begin{cor}\label{cor:sub_genus_up}
Let $\overline{k}$ satisfy the property that  the genus (resp. splitting genus) is trivial for any quaternion algebra over $\overline{k}$. Then the genus (resp. splitting genus) of any tame quaternion algebra over $k$ is trivial.
\end{cor}
\begin{remark}\normalfont
    Quaternion algebras sharing same separable maximal subfields are said to be \emph{totally separably linked} in \cite{quaternion_linkage}. It is shown in \cite[Corollay 6.2, Corollary 6.9]{quaternion_linkage} that  if $\overline{k}$ is perfect of characteristic $2$, then the genus of any  quaternion algebra over  $k=\overline{k}(\!(t)\!)$ is trivial. This is a special case of Corollary \ref{cor:sub_genus_up}, since  in this case ${}_2Br(\overline{k})$ is trivial (\cite[Chapter I, Theorem 1.3.7]{thelene_brauer}) and  any quaternion algebra over $k$ is tame (\cite[Theorem 1, Exercise 3, Chapter XII]{serre_local}). 
\end{remark}

One can recursively use   Corollary \ref{thm:sub_genus} to compute  genus of degree $p$ division algebras over fields of the form $k(\!(t_1)\!)(\!(t_2)\!)\cdots (\!(t_n)\!)$. 

\begin{example}\normalfont
Let $K = k(\!(t_1)\!)(t_2)\!)\cdots(\!(t_n)\!)$
where $k$ is a  field of characteristic $\neq p$. Let $D$ be a tame division algebra over $K$ be of index $p$.  Set $D^{(0)}:  = D$ and let $D^{(k)}$ denote the residue algebra of $D^{(k-1)}$, so that $D^{(n)}$ is a division algebra over $k$. Then by recursive application of  Corollary \ref{thm:sub_genus}, we get \begin{align*}
|gen^s([D])|   \leq \begin{cases}  
p-1 \text{~~if $D^{(n)}$ is a field~~}\\
 |gen^s([D]^{(n)})| \text{~~else}
\end{cases}
\end{align*}
\end{example}
\begin{example}
In the above example, taking $p=2$ and using the fact that genus of any tame quaternion over global field is trivial (\cite[\S 3.6]{rapinchuk_genus_unramified}), we see that the genus of  any quaternion algebra over $K = \mathbb{Q}(\!(t_1)\!)(\!(t_2)\!)\cdots(\!(t_n)\!)$ is trivial.
\end{example}

\subsection{Genus decomposition for $gen^s_{spl}$} 
We will now derive a formula for $gen^s_{spl}([D])$ for any $[D] \in SBr(k)$. Recall from  \S\ref{sec:cdvf} that we have  decomposition 
\begin{align*}
    D \sim I \otimes_k N
\end{align*}
in $SBr(k)$ where $I$ is inertial  and $N$ is  NSR over $k$ and $e_D = e_N$. 
\begin{lemma}
In the above decomposition, if $[D] \in {}_nSBr(K)$, so are $[I]$ and $[N]$
\end{lemma}
\begin{proof}
    It  suffices to show that $[N^{\otimes n}]$ is trivial. Now  $e_D| (ord[D])$ (\cite[Corollary 6.10]{wadsworth_henselian}) and therefore $e_D|n$. By  Lemma \ref{lem:exponent} $e_D =e_N = ord([N])$ and the result follows .
\end{proof}

\begin{thm} \label{thm:genus_cdvf} 
Let   $I$ be unramified and $N$ be  NSR division algebras over $k$.  Then 
\begin{align*}
    gen^s_{spl}([I \otimes_k N]) \subseteq \{ [I' \otimes_k N']~ |~ [I']  \in gen^s([I]) \text{~~  and~~} [N'] \in gen^s([N])\}
\end{align*}
\end{thm}
\begin{proof}
Let $D_1 \sim I_1 \otimes_k N_1$ and $D_2 \sim I_2' \otimes_k N_2$ where $[D_2] \in gen^s_{spl}([D_1])$. Since $I_1$ and $I_2'$ are unramified,  for $i=1,2$, we have $\chi_{D_i} = \chi_{N_i}$ (see \S\ref{sec:unramified}). As $D_1$ and $D_2$ have same splitting fields (and hence same subfields), by \cite[Lemma 2.5]{rapinchuk_genus_unramified}, we have $ker(\chi_{D_1}) = ker(\chi_{D_2})$. Therefore $e_{D_1}=e_{D_2}$ and   $\chi_{N_2} = \chi_{N_1^{\otimes j}}$ for some $j$ where $(j, e_{N_1}) =1$ (here we use the fact that $e_{N_i}=e_{D_i}$).  Hence $N_2 \sim N_1^{\otimes j} \otimes_k C$ where $[C] \in IBr(k)$. By replacing $I_2'$ with $I_2:=I_2'\otimes_k C$, we  get $D_2 \sim I_2 \otimes_k N_1^{\otimes j}$. Hence by Proposition \ref{prop:genusramified}, we have $D_2 \sim I_2\otimes N_2$ where $[N_2] \in gen^s([N_1])$. It remains to show that $[I_2] \in gen^s([I_1])$.\\
\indent Since $N_1$ is  NSR, it contains a totally ramified maximal subfield. By Appendix \ref{appendix:totram}, $N_1$  contains a totally ramified separable maximal subfield $L/k$. Therefore  $L$ is also a maximal subfield of  $N_2$.   So we have $D_i \otimes_k L \sim I_i \otimes_k L$. Now let $F/k$ be a finite separable maximal subfield of  $I_1$.  Note that since $I_1$ is unramified, $[\Gamma_F: \Gamma_k] =1$ (although $\overline{F}/\overline{k}$ need not be separable). By the proof of  \cite[Lemma 2.5.8]{field_arithmetic_book},  $F$ and $L$ are linearly disjoint (as subfields of a fixed algebraic closure of $k$). Therefore $F L \simeq F \otimes_k L$ is separable and  splits $D_1$.  Therefore it also  splits $D_2$.  Therefore 
\begin{align*}
    D_2 \otimes_k FL \sim  I_2 \otimes_k FL \sim (I_2 \otimes_k F) \otimes_F FL
\end{align*}
is split. But $FL/F$ is totally ramified. So $I_2 \otimes_k F$ is split by Lemma \ref{lem:totram}. In particular, $deg(I_2) \leq deg(I_1)$. Reversing the role of $I_1$ and $I_2$ we get $deg(I_1) \leq deg(I_2)$. Therefore $deg(I_1) = deg(I_2)$ and  $F$ is a maximal subfield of $I_2$. Since the argument is symmetric with respect to $I_1$ and  $I_2$, we get $I_2 \in gen^s([I_1])$.
\end{proof}
Now assume that $\overline{k}$ is perfect. Then we get a better estimate of $gen^s_{spl}$ as we show below.   For a finite separable extension $F/k$, denote by $\tilde{F}$ the maximal unramified subfield of $F/k$. We start with a lemma.

\begin{lemma}\label{lem:maxunramified}
Let $[I]\in IBr(k)$ and let $F/k$ be a finite separable extension.  Assume that $\overline{k}$ is perfect. Then $F$ splits $I$ if and only if $\tilde{F}$ splits $I$.
\end{lemma}
\begin{proof}
Suppose $F$ splits $I$. Note that $[I \otimes_k \tilde{F}] \in IBr(\tilde{F})$. Now
\begin{align*}
    I \otimes_k F &\simeq (I \otimes_k \tilde{F}) \otimes_{\tilde{F}} F
\end{align*}
Since  $\overline{k}$ is perfect, $\overline{F} = \overline{\tilde{F}}$ and therefore  $F/\tilde{F}$ is totally ramified.   By Lemma \ref{lem:totram}, $\tilde{F}$ splits $I$.
\end{proof}

\begin{thm} \label{thm:genus_cdvf2} 
Suppose $\overline{k}$ is perfect. Let   $I$ be unramified and $N$ be  NSR division algebras over $k$.  Then 
\begin{align*}
    gen^s_{spl}([I \otimes_k N]) \subseteq \{ [I' \otimes_k N']~ |~ [I']  \in gen^s_{spl}([I]) \text{~~  and~~} [N'] \in gen^s_{spl}([N])\}
\end{align*}
\end{thm}
\begin{proof}
The proof is similar to the proof of Theorem \ref{thm:genus_cdvf}, where $F$ is replaced with $\tilde{F}$. We leave the details to the reader.
\end{proof}
  Now  Theorem \ref{thm:genus_cdvf}, Theorem \ref{thm:genus_cdvf2} together with Proposition \ref{prop:inertial_residue}  and Proposition \ref{prop:genusramified} yields:

\begin{cor}\label{cor:splgenus}
 Let  $[D] \in SBr(k)$. Write $D \sim I \otimes_k N$  where  $I$ is inertial and $N$ is NSR. Then any element in $gen^s_{spl}([D])$ is of the form $[I' \otimes_k N^{\otimes j}]$ for some  $(j,e_D) =1$ where $[I'] \in gen^s([I])$. The algebra $I'$ is the (unique) inertial lift of some division algebra  whose class lies in $gen^s([\overline{I}])$. In particular
\begin{align*}
|gen^s_{spl}([D])| \leq |gen^s([\overline{I}])|\cdot \phi(e_D)
\end{align*}
where $\phi$ denotes the Euler's Totient function. If moreover $\overline{k}$ is perfect, then  $I'$ above is the (unique) inertial lift of some  division algebra whose class lies in $gen^s_{spl}([\overline{I}])$ and
\begin{align*}
|gen^s_{spl}([D])| \leq |gen^s_{spl}([\overline{I}])|\cdot \phi(e_D)
\end{align*}
\end{cor}
As an easy corollary for the case of ${}_2Br(k)$ we get:
\begin{cor} \label{cor:spl_gen_up}
 Let $\overline{k}$ satisfy the property that the genus is trivial for any element in  ${}_2Br(\overline{k})$. Then $gen^s_{spl}([D])$ is trivial for any $[D] \in {}_2SBr(k)$.
\end{cor}
\begin{proof}
By \cite[Corollary 6.10]{wadsworth_henselian}, for any $[D] \in {}_2Br(k)$,   $e_D|ord([D])$. Hence $e_D=2$ and the result follows from   Corollary \ref{cor:splgenus}.
\end{proof}

\begin{example}
 Since the genus of any element in ${}_2Br(\mathbb{Q})$ (this element is necessarily the class of some quaternion)  is trivial  (\cite[\S 3.6]{rapinchuk_genus_unramified}), by recursively applying  Corollary \ref{cor:spl_gen_up}, we see that the $gen^s_{spl}([D])$ is trivial for any $[D] \in {}_2SBr(\mathbb{Q}(\!(t_1)\!)(\!(t_2)\!)\cdots(\!(t_n)\!))$
\end{example}

\begin{question} \label{question:genus_up}\normalfont
Does the statement of Corollary \ref{cor:splgenus} hold if we replace $gen^s_{spl}$ with $gen$?
\end{question}

\begin{remark}\label{rmk:genus_up}
    If the answer to the above question is positive, then the stronger version of Corollary \ref{cor:spl_gen_up} where $gen^s_{spl}([D])$ is replaced with $gen^s([D])$ holds for any $[D] \in {}_2SBr(k)$. 
\end{remark}

\section{The Genus of quaternion algebras} \label{sec:qdvt}

In this section $K$ is an arbitrary field of $char \neq 2$ with a set of discrete valuations $V$. As before,  we denote the completion of $K$  with respect to $v$ and its residue field respectively by $K_v$ and $\overline{K_v}$.
\begin{defn}\normalfont  
For a division algebra $D$ over $K$, we say that $D$ (or its class  in $Br(K)$) is tame  with respect to $V$ if $D_v:=D \otimes_K K_v$ is tame for every $v \in V$.  We denote that set of tame elements  of $Br(K)$ with respect to $V$ by $SBr(K,V)$.
\end{defn}

\begin{remark}\normalfont \label{rmk:inertiallysplit}\normalfont
If the exponent of  $D_v$ is coprime to $char~\overline{K_v}$ or when $\overline{K_v}$ is perfect, $D_v$ is tame (\cite[Theorem 1, Exercise 3(b), Chapter XII]{serre_local}. In particular, every $[D] \in {}_2Br(K)$ is tame over $V$ if $char~\overline{K_v} \neq 2$ or if $\overline{K_v}$ is perfect for every $v \in V$. 
\end{remark}

For a field $K$, let 
\begin{align*}
   {}_2\Sha^{Br}(K, V)  := ker( {}_{2}Br(K) \rightarrow \prod_{v \in V}{}_2Br(K_v))
\end{align*}

\begin{thm}\label{thm:main}
Let $[Q]\in SBr(K,V)$ be the class of a quaternion division algebra over $K$. Suppose for every $v\in V$, the genus of  any  quaternion division algebra  over $\overline{K_v}$ is trivial. Then
\begin{align*}
gen^s([Q]) \subseteq [Q]+ {}_2\Sha^{Br}(K, V)
\end{align*}
In particular,
\begin{align*}
    |gen^s([Q])| \leq |{}_2\Sha^{Br}(K, V)|
\end{align*}
\end{thm}
\begin{proof}
Let $[P] \in gen^s([Q])$.   By \cite[Corollary 2.4]{RR_genus} ,  $[P_v] \in gen^s([Q_v])$ for every $v \in V$. Now by Corollary \ref{cor:sub_genus_up}, $[P_v] \sim [Q_v]$ for every $v \in V$. Therefore $[P \otimes_K Q] \in {}_2\Sha^{Br}(K, V)$ and the theorem follows. 
\end{proof} 

As an easy corollary of the theorem, we get:

\begin{cor}\label{cor:local_global}
Assume that $K$ satisfies local-global principle on the 2-torsion part of the Brauer group with respect to $V$ i.e, ${}_2\Sha^{Br}(K, V)$ is trivial. Suppose for every $v \in V$,  the genus of any class of quaternion division algebra in    $Br(\overline{K_v})$ is trivial, then the genus of any class of quaternion division algebra in  $SBr(K,V)$ is trivial. 
\end{cor}

\begin{cor}\label{cor:examples}
If $K$ is one of the following fields, the genus of any quaternion division algebra over $K$ is trivial
\begin{enumerate}
    \item  Higher local fields where the final residue field has characteristic $\neq 2$
    \item  Iterated Laurent series $k(\!(t_1)\!)(\!(t_2)\!)\cdots(\!(t_n)\!)$ (resp. their finite extensions) where $char~k \neq 2$ and every quaternion division algebra over $k$ (resp. every finite extension of $k$) has trivial genus
    \item Function fields of  one variable over fields in (1)
    \item Function fields of one variable any real closed field 
    \item Function fields of  one variable  over fields in (2) where  for any curve $C$ over $k$ (where $k$ is as in (2)),  every quaternion algebra over every finite extension of $k$ and  over $k(C)$ has trivial genus (for example, $k(\!(t_1)\!)(\!(t_2)\!)\cdots(\!(t_n)\!)(C)$  where one can take $k$ to be  any  real closed field by  (4))
\end{enumerate}
\end{cor}
\begin{proof} For each $K$ as above, the proof below considers the  discrete valuations $V$  where residue fields are of  characteristic $\neq 2$, so  all the division algebras over these fields are tame by Remark \ref{rmk:inertiallysplit}. \\
 \indent For (1) set $V = {v}$  to be the canonical discrete valuation on  $K$. Then by induction on $cdv.dim(K)$ and observing that the final residue field is finite, the result follows from Corollary \ref{cor:local_global}.  The proof of (2) is similar by using induction on $n$. \\
\indent For  (3), let $K = F(C)$  where $F$ is a field in (1). We will use induction on $d:=cdv.dim(F)$   of the higher local field $F$. When $d = 0$, $F$ is a finite field and $K$ is a global field and hence the genus of any quaternion over $K$ is trivial (\cite[\S 3.6]{rapinchuk_genus_unramified}). Now assume that the property is satisfied by $K$ for every $F$ with  $cdv.dim(F) \leq d$. Let $cdv.dim(F) = d+1$ and let  $V$ be  discrete valuations corresponding to  codimension 1 points  in a regular proper model of $K$ over $\mathcal{O}_F$. Then by \cite[Theorem 4.2 and Theorem 4.3 (ii)]{parimala_local_global}, ${}_2\Sha^{Br}(K, V)$ is trivial.  Moreover, in this case, the residue fields corresponding to $v \in V$ are either finite extensions of $F$ which is again a higher local field (\S\ref{sec:higher_local}) or of the form  $F'(C')$ with $cdv.dim(F') \leq d$. So by (1) and  induction hypothesis, the result now follows from Corollary \ref{cor:local_global}. \\
\indent For (4), let $C$ be a smooth projective curve over a real closed field $R$ with function field  $K$. Such a curve exists by \cite[Theorem 53.2.6 and Lemma 53.2.8]{stacks-project}. By  in \cite[Theorem 2.3.1]{parimala_real}, the specialization  map  $Br(C) \rightarrow \prod_{P \in C(R)} Br(k_P)$   is injective where $k_P$ is the residue field at $P$. Taking $V$ to be the discrete  valuations corresponding to the points in $C(R)$, Lemma \ref{lem:specialize}  yields triviality of ${}_2\Sha^{Br}(K, V)$ and so the result follows from  Corollary \ref{cor:local_global}.\\
\indent The proof of (5) is again using an induction argument similar to (3). 
\end{proof}

\subsection{Curves over global fields with rational point} \label{sec:elliptic}
Let $C$ be a smooth projective geometrically integral curve over a global field $k$ with a $k$-rational point. In this section, we discuss the relation between the genus of  quaternion algebras over the function field of $C$   and the $2$-torsion  subgroup of the Tate-Shafarevich  group of the Jacobian of $C$.  We deduce that the size of the  genus is bounded above by the size of the  $2$-torsion component of the Tate-Shafarevich  group of the Jacobian.  Then we specialize to the case of elliptic curves and  demonstrate that this bound is better than the one known before.

Let $C$ be a smooth projective geometrically integral curve over a global field $k$ with $C(k) \neq \emptyset$ and let $J_C$ denote its Jacobian.   Recall that the Tate-Shafarevich  group of $J_C$, denoted by  $\Sha(J_C)$ is defined as
\begin{align*}
    \Sha(J_C) = ker(H^1(k, J_C(k^{sep}) \rightarrow  \prod_{v \in P(k)} H^1(k_v, J_C(k^{sep}_v))
\end{align*}
where $P(k)$ denotes the set of places of $k$. Let $C_v: = C \times_{Spec~k} Spec~k_v$ and let $\Sha^{Br}(C)$  be the kernel of the local global map on the Brauer group of $C$ with respect to $P(k)$, i.e,
\begin{align*}
    \Sha^{Br}(C) := ker( Br(C) \rightarrow \prod_{v \in P(k)} Br(C_v))
\end{align*}
By \cite[\S 2(B)]{parimala_sujatha}, we have the following isomorphism\footnote{Although the isomorphism  (\ref{eqn:parimala_sujatha}) is shown for number fields in \cite{parimala_sujatha}, the same arguments show that the isomorphism holds for global fields as well.}:
\begin{align}\label{eqn:parimala_sujatha}
    \Sha^{Br}(C) \simeq \Sha(J_C)
\end{align}
Let $S:=Spec~\mathcal{O}_k$ where  $\mathcal{O}_k$ is the ring of integers  of $k$ when $char~k =0$ and  a smooth complete curve over its field of constants with function field $k$  when $char~k \neq 0 $. Suppose $\mathcal{C} \rightarrow S$ is a regular proper  model for $C$. Let $V_{\mathcal{C}}$ denote the valuations on $k(C)$ corresponding to codimension one points $\mathcal{C}^{(1)}$ of $\mathcal{C}$. We define (see \S Appendix \ref{sec:appendix} for details),
\begin{align*} 
    \Sha^{Br}(C/\mathcal{C}) := \Sha^{Br}(k(C), V_{\mathcal{C}}) =  ker(Br(k(C)) \rightarrow \prod_{x \in \mathcal{C}^{(1)}} Br(k(C)_x)
\end{align*}
 where $k(C)_x$ denotes the completion of $k(C)$ with respect to the valuation corresponding to the codimension $1$ point $x \in \mathcal{C}^{(1)}$.

Then by the isomorphism in (\ref{eqn:parimala_sujatha}) together with Theorem \ref{thm:comparison}, we conclude

\begin{thm}\label{thm:sha_E}
 With notations as above, we have
\begin{align*}
    \Sha^{Br}(C/\mathcal{C}) \simeq \Sha(J_C)
\end{align*}
In particular,  when $C=E$ is an elliptic curve with a regular proper  model $\mathcal{E}$,
\begin{align*}
    \Sha^{Br}(E/\mathcal{E}) \simeq \Sha(E)
\end{align*}
\end{thm}
\begin{remark}\normalfont
Recall that $index = period$ for any element in  $\Sha(E)$ (\cite[Theorem 1.3]{cassels4}). Therefore the same holds for  $\Sha^{Br}(E/\mathcal{E})$. In particular, the elements in  ${}_2\Sha^{Br}(E/\mathcal{E})$ correspond to quaternions.
\end{remark}
 The residue fields of $K:=k(C)$ corresponding  to $V_{\mathcal{C}}$ are global fields and therefore the genus of any class of quaternion division algebra in  ${}_2Br(\overline{K_v})$ is trivial for every $v \in V$ (\cite[\S 3.6]{rapinchuk_genus_unramified}).  So by Theorem \ref{thm:main} and  Theorem \ref{thm:sha_E}, we conclude 

\begin{thm}\label{thm:gen2}
With notations as above, let $[Q] \in {}_2SBr(K,V)$ be the class of a quaternion algebras  where $K=k(C)$ and $V = V_{\mathcal{C}}$. Then 
\begin{align*}
    |gen^s([Q])| \leq |{}_2\Sha(J_C)|
\end{align*}
Therefore when $C=E$ is an elliptic curve,  we have 
\begin{align}\label{eqn:gen_sha}
    |gen^s([Q])| \leq |{}_2\Sha(E)|
\end{align}
In particular, we get $|gen^s([Q])|$ is trivial whenever ${}_2\Sha(E)$ is trivial.
\end{thm}

\begin{remark}\normalfont
Comparing with \cite[Theorem 4.1, Corollary 4.11]{rapinchuk_genus_unramified}, we see that the  bound in (\ref{eqn:gen_sha}) gives a better estimate of the genus of any class of  quaternion algebra $[Q]\in SBr(K,V)$ since for a regular proper model $\mathcal{E}$ of $E$, $\Sha(E) \simeq \Sha^{Br}(E/\mathcal{E}) \subseteq Br(E)_{V_0 \cup V_1}$ where $V_0$ and $V_1$ are the sets of valuations on $K$ given in   \cite[\S4]{rapinchuk_genus_unramified}.  See also Example \ref{example:elliptic}.
\end{remark}

One can extensively use the arithmetic properties of a given  elliptic curve $E$ over a number field  together with Theorem \ref{thm:gen2}  to compute  bounds on  the genus. We will demonstrate this  below. \\
\indent For the rest of the section, $E$ denotes  an elliptic curve over a number field $k$  and $[Q]$ is an arbitrary class of quaternion division algebra in $SBr(K,V)$ where  $K=k(E)$, $V= V_{\mathcal{E}}$ for a  regular proper model $\mathcal{E}$  of $E$.

\begin{example}\normalfont  \label{example:elliptic}
Let $E$  be given by $y^2 = x^3 - x$ over $\mathbb{Q}$. We have ${}_2\Sha(E) =0$ (one can verify using MAGMA) and therefore from Theorem \ref{thm:gen2}, we conclude that  $gen^s([Q])$ is trivial. Compare with \cite[Example 4.12]{rapinchuk_genus_unramified}.
\end{example}

\begin{example}\normalfont
Let $p$ be an odd prime and let $E$ be given by $y^2 = x^3 +px$ over $\mathbb{Q}$. By Theorem \ref{thm:gen2} and \cite[Proposition X.6.2(c)]{silverman_elliptic}, we get

\begin{align*}
    |gen^s([Q])| \leq \begin{cases}1 \text{~~~if~~~} p \equiv 7,11 (mod~16)\\
    2 \text{~~~if~~~} p \equiv 3,5, 13, 15 (mod~16)\\
    4  \text{~~~if~~~} p \equiv 1,9 (mod~16)
    \end{cases}
\end{align*}
\end{example}
We will now explicitly show triviality of the genus of some quaternion division algebras  over $\mathbb{Q}(E)$  using  arithmetic properties of $E$ even when ${}_2\Sha(E)$ is not trivial. Given an isogeny  $\phi: E \rightarrow E'$ over $k$ with  $\hat{\phi}$ denoting  the dual isogeny, let $S^{\phi}(E)$ and $S^{\hat{\phi}}(E')$ denote the respective Selmer groups. 

We have an exact sequence of groups \cite[Chapter X, Theorem 4.2] {silverman_elliptic}
\begin{align*}
    0 \rightarrow E'(k)/\phi(E(k)) \rightarrow S^{\phi}(E/k) \rightarrow {}_{\phi}\Sha(E/k) \rightarrow 0 
\end{align*}
where  ${}_{\phi}\Sha(E/k)$ denotes the $\phi$-torsion subgroup of $\Sha(E/k)$.  Let $S \subset P(k)$ denote  the union of the set of infinite places, the set of finite primes
at which $E$ has bad reduction, and the set of finite primes dividing $2$.  Recall \cite[Chapter X, Proposition 1.4 and Proposition  4.9]{silverman_elliptic} that if $\phi$  is of degree $2$, $S^{\phi}(E/k)$ and $ S^{\hat{\phi}}(E'/k)$ are subgroups of $ k(S,2)$ where 
\begin{align*}
    k(S,2) = \{b \in k^*/(k^*)^2 : ord_v(b) \equiv 0 (mod~2) \forall v \notin S\}
\end{align*}

\begin{example} \normalfont
Let $E$ be given by $y^2 = x^3 -113x$ over $k:=\mathbb{Q}$. There is $2$-isogeny (\cite[X.6]{silverman_elliptic})
\begin{align*}
    \phi: E \rightarrow E'
\end{align*}
where $E'$ is given by $y^2 = x^3 + 452x$.   Let $E'^{(d)}$ denote the quadratic twist of $E'$ by $d$. We observe that $E'^{(-1)} \simeq E'$. Let $l:=\mathbb{Q}(\sqrt{-1})$.  One can verify that   $E'_{tors}(l) = E'_{tors}(k)$ and $rank(E'(k))=0$ using Magma.  Moreover  by \cite[Chapter X, Exercise 10.16]{silverman_elliptic}, $rank(E'(l)) = rank(E'(k)) + rank(E'^{(-1)}(k)) = 2\cdot rank(E'(k)) = 0$.   Therefore, $E'(l) = E'(k)$. \\
\indent Now let $\hat{\phi}$ denote the dual isogeny. From the  computations given in \cite[Lemma 3.4]{parimala_sujatha}, we get
\begin{align*}
S^{\phi}(E/k) = \{ 1, 113, 2, 226\} (~mod~ (k^*)^{2})\\
S^{\hat{\phi}} (E'/k) = \{\pm 1 \pm 113\} (~mod~ (k^*)^{2})
\end{align*}
We have $|{}_{\phi}\Sha(E/k)| = |{}_{\hat{\phi}}\Sha(E'/k)| = 2$ and ${}_{2}\Sha(E'/k) = {}_{\phi}\Sha(E/k) \oplus {}_{\hat{\phi}}\Sha(E'/k) $. The non-trivial element in ${}_{\phi}\Sha(E/k)$ and ${}_{\hat{\phi}}\Sha(E'/k)$ are respectively given by the principal homogeneous spaces corresponding to ${2}$ and ${-1}$. Let $D_{2}$ and $D_{-1}$ denote the respective (non-split) quaternion algebras under the isomorphism (\ref{eqn:parimala_sujatha}).
Let us compute  $gen^s([D_{-1}])$.  It is easy to see that $D_{-1}$ is split by $l$ since the class of $-1$ is trivial in $S^{\hat{\phi}}(E'/l)$.  We also have the following commutative diagram where the vertical maps are induced by restrictions
\begin{equation*}
\begin{tikzcd} 
  0 \arrow[r] & E'(k)/\phi(E(k)) \arrow[d, ""] \arrow[r, "i_k"] & S^{\phi}(E/k)\arrow[d, "\theta"] \arrow[r, ""] & {}_{\phi}\Sha(E/k) \arrow[d, ""] \arrow[r] & 0 \\
  0 \arrow[r] & E'(l)/\phi(E(l)) \arrow[r, "i_l"] & S^{\phi}(E/l) \arrow[r, ""] & {}_{\phi}\Sha(E/l) \ar[r] & 0
\end{tikzcd}
\end{equation*}

\noindent   By Theorem \ref{thm:main}, 
\begin{align*}
    gen^s([D_{-1}] \subseteq \{ [D_{-1}], [D_{2}], [D_{-1} \otimes D_{2}]\}
\end{align*}
We will show that  $gen^s([D_{-1}])$ is trivial. So it suffices to show that $D_{2}$ is not split by $l$ or equivalently that $\theta(2)$ is not in the image of $i_l$.  Now suppose $\theta(2)$ is  in the image of $i_l$, say $i_l(P) = \theta(2)$ for some $P \in E'(l) = E'(k)$. Then
\begin{align*}
   \theta(i_k(P)) &= \theta(2)\\
  \implies 2\cdot i_k(P) &\in Ker(\theta) = S^{\phi}(E/k) \cap \{1,-1\} (mod~ (k^*)^2) = 1 (mod~ (k^*)^2)
\end{align*}
Therefore  $i_k(P) =2 (mod~(k^*)^2)$. This contradicts the fact that $2$ is non-trivial in  ${}_{\phi}\Sha(E/k)$ and concludes that  $gen^s([D_{-1}]) = \{ [D_{-1}]\}$.
\end{example}
For a quadratic extension $l/k$, the techniques used in the above example can be generalized to compute genus of classes  in $Br(l/k)$ whenever $E(l) = E(k)$. We demonstrate this below. \\
\indent We will assume that $E$ is split, i.e, $E$ is given by a Weierstrass equation $y^2 = f(x)$ where $f(x)$ has three distinct roots over $k$. In this case $E[2] \subseteq E(k)$ and the $2$-Selmer group $S^{(2)} (E/k)$ is a subgroup of $k(S,2) \times k(S,2)$. Now let $l=k(\sqrt{a})$ be a quadratic extension. Denote by 
\begin{align*}
    \theta : S^{(2)} (E/k) \rightarrow S^{(2)} (E/l)
\end{align*}
the natural map induced by restriction. Identifying $S^{(2)} (E/k)$ as a subgroup of $k(S,2) \times k(S,2)$, we have 
\begin{align} \label{eqn:selmer_kernel}
    H:=ker (\theta) = S^{(2)} (E/k) \cap \{(1,1), (a,1), (1,a), (a,a) \} (mod ~(k^*)^2)
\end{align}

Given $ \gamma \in S^{(2)}(E/k)$, let $D_{\gamma}$ denote the image of $\gamma$ under the composition
\begin{align*}
    S^{(2)}(E/k) \rightarrow {}_{2}\Sha(E/k) \xrightarrow{\simeq} {}_2\Sha^{Br}(E/k)
\end{align*}
where the last arrow comes from the isomorphism (\ref{eqn:parimala_sujatha}).

\begin{lemma}
 Assume that $E$ is split. Let $l=k(\sqrt{a})$ be a quadratic extension such that $E(l) = E(k)$. If $[Q]$ is split over $l(E)$, then 
 \begin{align*}
     gen^s([Q]) \subseteq \{[Q \otimes D_{\gamma}] | \gamma \in H \}
 \end{align*}
 where $H$ is given by (\ref{eqn:selmer_kernel}).  In particular, 
 \begin{align*}
     |gen^s([Q])| \leq 4
 \end{align*}
 and $gen^s([Q])$is trivial if $H$ is trivial.
\end{lemma}
\begin{proof} We have the following commutative diagram of exact sequences where the vertical arrows are induced by restrictions.

\begin{equation*}
\begin{tikzcd} 
  0 \arrow[r] & E(k)/2E(k) \arrow[d, ""] \arrow[r, "i_k"] & S^{(2)}(E/k)\arrow[d, "\theta"] \arrow[r, ""] & {}_2\Sha(E/k) \arrow[d, ""] \arrow[r] & 0 \\
  0 \arrow[r] & E(l)/2E(l) \arrow[r, "i_l"] & S^{(2)}(E/l) \arrow[r, ""] & {}_2\Sha(E/l) \ar[r] & 0
\end{tikzcd}
\end{equation*}
By Theorem \ref{thm:main} and the above diagram, any element in   $gen^s([Q])$ is of the form  $[Q \otimes D_{\gamma}]$ for some   $\gamma \in S^{(2)}(E/k)$. So it suffices to show that $\gamma \in i_k(P)H$ for some $P \in E(k)$.  Since $[Q]$ is split over $l(E)$, for any $[Q \otimes D_{\gamma}] \in gen^s([Q])$, $D_{\gamma}$ is split over $l(E)$ which implies $\theta(\gamma )= i_l(P)$ for some $P \in E(l) = E(k)$. Therefore by commutativity of the above diagram, we have
\begin{align*}
\theta(i_k(P)) = \theta(\gamma) \\
\implies \gamma \in i_k(P) H
\end{align*}
which finishes the proof.
\end{proof}

\begin{remark}
    If the answer to Question \ref{question:genus_up} is positive, then by Remark \ref{rmk:genus_up}, all the results and examples of this section hold if quaternions are replaced by division algebras of exponent two 
\end{remark}

\appendix
\section{Curves over global fields}\label{sec:appendix}
Throughout this section $k$ denotes a global field. Let $S$ denote the Dedekind scheme given by  $Spec~\mathcal{O}_k$ where  $\mathcal{O}_k$ is the ring of integers  of $k$ when $char~k =0$ and  a smooth complete curve over its field of constants with function field $k$  when $char~k \neq 0 $. Let $C$ denote a smooth geometrically integral projective curve over $k$. By a theorem of Lipman (\cite{liu}, \S10.1.1  and  Corollary 8.3.51) there exists a regular projective model  for $C$ over  $S$ i.e., there exists  a regular fibered surface $\mathcal{C}$ over  $S$ such that the generic fiber is isomorphic to $C$ i.e, $C\simeq  \mathcal{C} \times_S Spec~k$. Since the map $\mathcal{C} \rightarrow S$ is dominant, $k(\mathcal{C}) \simeq k(C)$.  Therefore, the codimension 1 points  $\mathcal{C}^{(1)}$ in $\mathcal{C}$ give rise to discrete valuations on $k(C)$. For $x \in \mathcal{C}^{(1)}$, let $k(C)_x$ denote the completion of $k(C)$ with respect to this valuation.  Let 
\begin{align*} 
    \Sha^{Br}(C/\mathcal{C}) := ker(Br(k(C)) \rightarrow \prod_{x \in \mathcal{C}^{(1)}} Br(k(C)_x)
\end{align*}
be the kernel of the local global map on the Brauer group with respect to the valuations from the codimension 1 points on $\mathcal{C}$.  One also has the local global map on $Br(C)$ with respect to the places  $P(k)$ of $k$ whose kernel will be denoted by $\Sha^{Br}(C)$.
\begin{align*}
    \Sha^{Br}(C) := ker( Br(C) \rightarrow \prod_{v \in P(k)} Br(C_v))
\end{align*}
Since $C$ is a smooth integral curve, the natural map $Br(C) \rightarrow Br(k(C))$ is injective by Corollary 1.10 in \cite{brauer2}. So we consider $Br(C)$ as a subgroup of $Br(k(C))$.  One may ask how are the groups $\Sha^{Br}(C/\mathcal{C})$ and $\Sha^{Br}(C)$ related as subgroups of $Br(k(C))$. The main result of this section is the following theorem:

\begin{thm} \label{thm:comparison}
 With notations as above,  
 \begin{align*}
     \Sha^{Br}(C/\mathcal{C}) = \Sha^{Br}(C)
     \end{align*}
as subgroups of $Br(k(C))$.  
\end{thm}
In order to prove the theorem we need a few lemmas. Although the following lemmas are well known, we could not find any explicit references stating these results. Therefore, we include them here for completeness and for the reader's convenience.
\begin{lemma} \label{lem:extn_places}
Let $x \in X$ be a closed point on a scheme $X$ locally of finite type over $k$ with residue field $k(x)$. For a place $v \in P(k)$, let $P_v(k(x))$ denote the valuations on $k(x)$ extending $v$. Then the set of completions $\{ k(x)_w | w \in P_v(k(x))\}$ is equal to the residue fields of the points lying above $x$ under the base change map 
\begin{align*}
    \phi_v: X_v:= X \times_k k_v \rightarrow X
\end{align*}
\end{lemma}
\begin{proof}
The closed point $x$ corresponds to $Spec~k(x) \hookrightarrow X$ where $k(x)$ is a finite  extension of $k$. The residue fields of  points lying above $x$  in $X_v$ are precisely  $\{L_i\}$  where  $k(x) \otimes_k k_v \simeq \prod L_i$. When $k(x)/k$ is separable,  by  \cite[Propositon 8.1 and 8.2]{milneANT},  $\{L_i\}$ is the set of completions of $k(x)$ with respect to possible extensions of $v$ to $k(x)$ which proves the lemma. When $k(x)/k$ is purely inseparable, $v$ extends uniquely   to $k(x)$ (\cite[Theorem 4.1]{defect_kuhlmann}) and the lemma follows. The general case is derived using the above two cases. 
\end{proof}

Let $X$ be a regular integral scheme   with function field $k(X)$. For a codimension 1 point  $x \in X^{(1)}$, let $k(x)$ denote the residue field of $x$ and  $k(X)_x$ denote the completion of $k(X)$ with respect to the discrete valuation associated to $x$.  For  $\alpha \in Br(X)$ and $x \in  X^{(1)}$,  let $\alpha_x$ and $\alpha(x)$ denote respectively the image of $\alpha$ under the restriction maps given by 
\begin{align*}
     r_{X,x}: Br(X) \rightarrow Br(k(X))
    &\rightarrow Br(k(X)_x) \\
    \alpha &\mapsto \alpha_x\\
    s_{X,x}: Br(X) \rightarrow Br(\mathcal{O}_{X,x}) &\rightarrow Br(k(x)) \\
    \alpha &\mapsto \alpha(x)
\end{align*}

\begin{lemma} \label{lem:specialize}
With notations as above, for  $\alpha \in Br(X)$, $\alpha_x = 0 \iff \alpha(x) = 0$
\end{lemma}
\begin{proof}
This follows from the commutativity of the following diagram where  $\hat{\mathcal{O}}_{X,x}$ denotes the completion of $\mathcal{O}_{X,x}$.  The maps $\imath$, $\jmath$ are injective (Corollary 1.10 in \cite{brauer2}) and $q$ is an isomorphism (\cite[Theorem 31]{azumaya}).

\begin{center}
\begin{tikzpicture}[thick] 
  \node (A) {$Br(k(X))$};
  \node (B) [right of=A] {$Br(k(X)_x)$};
  \node (D) [below of=B] {$Br(\hat{\mathcal{O}}_{X,x})$};
\node (E) [left of = A] {$Br(X)$};
  \node (C) [below of=E] {$Br(\mathcal{O}_{X,x})$};
\node (F) [below of = C] {$Br(k(x))$};

\draw[->] (E) to node {} (C);
  \draw[->] (A) to node {} (B);
\draw[right hook->] (E) to node {$\imath$} (A);
  \draw[->] (C) to node   {}(D);
  \draw[right hook->] (D) to node {$\jmath$} (B);
   \draw[->] (C) to node {} (F);
    \draw[->] (D) to node {$q$}  node[swap] {$\cong$}(F);
    \draw[->] (E) to[bend right] node {$r_{X,x}$} (B);
    \draw[->] (E) to[bend right =50] node {$s_{X,x}$} (F);
 \end{tikzpicture}
 \end{center}
\end{proof}

We are now ready to prove the main theorem of this section.\\

\noindent \emph{Proof of Theorem \ref{thm:comparison}:} First we will show the inclusion
\begin{align*}
     \Sha^{Br}(C/\mathcal{C}) \supseteq \Sha^{Br}(C)
     \end{align*}
 Let $\alpha \in Br(C)$ be such that $\alpha_v = 0$  for every  $v \in P(k)$, where $\alpha_v$ is the image of $\alpha$ under $Br(C) \rightarrow Br(C_v)$. We need to show that $\forall x \in \mathcal{C}^{(1)}$, $\alpha_x$ is zero. \\
 \indent Suppose $x$ corresponds to a vertical divisor in the fiber corresponding to $v\in P(k)$. Then, $k_v(C) \subseteq k(C)_x $ which yields the required result.\\
 \indent On the other hand, if $x$ corresponds to a horizontal divisor of $\mathcal{C}$, then $x$ is a closed point in $C$. By  Lemma \ref{lem:specialize}, we need to show that $\alpha(x) \in Br(k(x))$ is trivial.  Let $w$ be a valuation on in $k(x)$ that restricts to the valuation $v$ on $k$.  Denote by $x_w \in \mathcal{C}_v$, the point corresponding to $k(x)_w$ (Lemma \ref{lem:extn_places}). Then  $\alpha(x)_w= \alpha_v(x_w)=0$ since $\alpha_v = 0$  for every  $v \in P(k)$. Therefore $\alpha(x)$ is in the kernel of the local-global map
 \begin{align*}
    Br(k(x)) \rightarrow \prod_{v \in P(k(x))} Br(k(x)_v)
\end{align*}
which is trivial by the classical Albert-Brauer-Hasse-Noether Theorem for the number field case and its generalization by Hasse  for global fields (\cite[Corollary 6.5.4]{gille_sam}). \\
\indent We will now show 
\begin{align*}
     \Sha^{Br}(C/\mathcal{C}) \subseteq \Sha^{Br}(C)
\end{align*}
Suppose $\alpha \in \Sha^{Br}(C/\mathcal{C})$.  Then note that $\alpha$ is unramified with respect to every $x\in \mathcal{C}^{(1)}$ and hence $\alpha \in Br(\mathcal{C})$.  Let $v \in P(k)$ be non-archimedean.  Now $\mathcal{C}_v := \mathcal{C} \otimes_{\mathcal{O}_k} \mathcal{O}_v$ is a regular projective model for $C_v$ over $\mathcal{O}_v$ ( \cite[\href{https://stacks.math.columbia.edu/tag/0BG4}{Tag 0BG4}, Theorem 54.11.2]{stacks-project}). Then $\alpha \otimes_{\mathcal{C}} C_v \simeq (\alpha \otimes_{\mathcal{C}} \mathcal{C}_v) \otimes_{\mathcal{C}_v} C_v = 0$ since Brauer group of a regular proper curve over a complete discrete valued ring  with finite residue field is trivial (\cite[Theorem 3.1 and Remark 2.5(b)]{brauer3}). \\
\indent Now let $v$ be archimedean. If $k_v \simeq \mathbb{C}$, then  since $Br(C_v) = 0$ by Tsen's theorem, $\alpha \otimes_{\mathcal{C}} C_v =0$. Suppose $k_v \simeq \mathbb{R}$. Let $\tilde{k_v}$ denote the real closure of $k$ with respect to $v$.  It suffices to show that $\tilde{\alpha}_v : =\alpha \otimes C_{\tilde{k_v}} = 0$. By \cite[Theorem 2.3.1]{parimala_real}, this is equivalent to showing that $\tilde{\alpha}_v(y)=0$ for every real closed point $y \in  C_{\tilde{k_v}}$.  Now let $x$ be the image of $y$ under the canonical map $ C_{\tilde{k_v}} \rightarrow C$. 
Note that $x$ is a closed point of codimension 1.  Since $\alpha \in  \Sha^{Br}(C/\mathcal{C})$, $\alpha_x =0$. Hence by Lemma \ref{lem:specialize}, $\alpha(x) =0$. Therefore $\tilde{\alpha}_v(y) \simeq \alpha(x) \otimes_{k(x)} k(y) = 0$.

\section{Totally ramified separable subfields in NSR algebras} \label{appendix:totram}

Let $k$ be a complete discrete valued field with value group $\mathbb{Z}$ and residue field $\overline{k}$. Let $v$ denote the valuation on $k$. Let $N$ be a  NSR division algebra over $k$.   The main goal of this section is to show that $N$ contains a totally ramified separable maximal subfield.  If $char~k =0$ or is coprime to $deg(N)$, the every subfield of $N$ is separable and the claim follows since $N$ contains a totally ramified maximal subfield by definition. So we may assume that $(char~k, deg(N)) = p$. \\
 \indent Recall  from Proposition \ref{prop:nsr} that $N$ is a cyclic algebra and hence contains some cyclic maximal subfield. The next theorem sheds some light on the ramification. 
 \begin{thm}\label{thm:app:totram}
    Let $D$ be a NSR division algebra over $k$ containing a totally ramified purely inseparable maximal subfield. Then $D$ contains a totally ramified cyclic maximal subfield.
\end{thm}
\begin{proof}
Since $D$ contains a purely inseparable maximal subfield,  $deg(D) = p^n$  for some $n$ where $p=char~k$. Let $F$ be a totally ramified purely inseparable maximal subfield of  $D$. Then there exists an element $y\in F$ with $v(y) = \frac{1}{p^n}\mathbb{Z}$. Therefore $y$ generates $F$ over $k$ i.e, $F\simeq k(y)$ where $y^{p^n} = a$ for some $a \in k^*$ with $(v(a),p) =1$.    By  (\cite[Chapter VII, Theorem 26]{albert}), $D$ is isomorphic to the symbol algebra $[\omega, a)$  where $\omega  = (\omega_1, \omega_2, \cdots, \omega_n) \in W_n(k)$, the group of truncated Witt vectors of length $n$.  By symbol manipulation techniques (\cite[Satz 15, 16]{witt}, \cite{teichmuller}, \cite[Proposition 1]{mamm_merk}),  we may assume that $v(a)<min(0, v(\omega_1))$ and 
\begin{align*}
    D \cong [\omega', a):=[(\omega + (a,0,0 \cdots 0) , a)]
     \end{align*}
 where $ \omega' = (\omega_1 +a, \omega_2', \cdots, \omega_n') \in W(k)$ for some $\omega_2', \cdots, \omega_n' \in k$. Consider the Artin-Schreier-Witt extension  $L$ corresponding to $\omega'$. Then $L \simeq  k(x_1, x_2, \cdots, x_n)$ where $(x_1^p, x_2^p, \cdots, x_n^p)- (x_1, x_2, \cdots, x_n) = \omega'$  in $W(L)$. We claim that $L$ is totally ramified which will yield the theorem. To see this, let $\overline{L}$ denote the residue field of $L$. Suppose $L/k$ is not totally ramified, then    $\overline{L} \subseteq \overline{D}$ is a non-trivial extension of $\overline{k}$. But $D$ is NSR, so $\overline{D}$ is  a field that is a cyclic extension (\cite[Lemma 5.1]{wadsworth_henselian}) of $\overline{k}$ of degree $p^n$. Let $m \subset \overline{L} \subseteq \overline{D}$ be the unique degree $p$ extension over $\overline{k}$ and let $M \subset L$ be the inertial lift  of $m$ \cite[Chapter III, \S5, Corollary 2]{serre_local}).  Therefore $M$ is an unramified  degree $p$ extension of $k$ inside $L$. But $L/k$ is cyclic and contains the unique degree $p$ extension $k(x_1)$ defined by 

\begin{align*}
x_1^p - x_1 = \omega_1 + a
\end{align*}
By the assumption on $v(a)$, we have  $v(x_1) = \frac{1}{p}v(a)$. Since  $(v(a), p ) =1$, the extension  $k(x_1)$ is ramified over $k$ leading to  contradiction. Therefore $L/k$ is totally ramified. 
 \end{proof}

 \begin{cor}
     Let $N$ be a  NSR division algebra over $k$. Then $N$ contains a totally ramified separable maximal subfield.
 \end{cor}
 \begin{proof}
     Let $L \subset N$ be a totally ramified maximal subfield. Then there is a tower of extensions $k \subseteq F\subseteq L$ where $F/k$ is separable and $L/F$ is purely inseparable (\cite[\href{https://stacks.math.columbia.edu/tag/030K}{Lemma 030K}]{stacks-project}). Note that both $F/k$ and $L/F$ are totally ramified. Let $D:= C_N(F)$ be the centralizer of $F$ in $N$. Then $D$ is a division algebra over $F$  and $D \simeq N\otimes_k F$(\cite[\S13.3, Lemma]{pierce_assoc}). Let $K/k$ be an unramified splitting field of $N$. Then $FK/F$ is  unramified (\cite[Chapter II, \S4, Proposition 8(ii)]{lang_number_theory}) and splits  $D$. Therefore $D$ is tame. Now $D$ contains the totally ramified  purely inseparable  subfield $L$ which is a maximal subfield by the Double Centralizer Theorem (\cite[\S12.7 Theorem]{pierce_assoc}). Therefore by Theorem \ref{thm:app:totram}, $D$ contains a cyclic totally ramified maximal subfield $M/F$. Since $F/k$ is separable, $M/k$ is a totally ramified separable maximal subfield in $N$.
 \end{proof}
\section*{Acknowledgements}
The author acknowledges the support of the DAE, Government of India, under Project Identification No. RTI4001. She is grateful to A. Rapinchuk for the numerous email conversations on this topic and his valuable feedback. She also thanks R. Parimala  for pointing to the relevant literature and  her useful comments and suggestions on the Appendix. Finally, she thanks the referee for the corrections and improvements. 

\nocite*{}
\bibliographystyle{alpha}
\bibliography{ref_quaternions}

\newcommand{\etalchar}[1]{$^{#1}$}
\begin{thebibliography}{KMR{\etalchar{+}}22}

\bibitem[Alb61]{albert}
A.~Adrian Albert.
\newblock {\em Structure of algebras}, volume Vol. XXIV of {\em American
  Mathematical Society Colloquium Publications}.
\newblock American Mathematical Society, Providence, RI, 1961.
\newblock Revised printing.

\bibitem[Azu51]{azumaya}
Gor\^{o} Azumaya.
\newblock On maximally central algebras.
\newblock {\em Nagoya Math. J.}, 2:119--150, 1951.

\bibitem[Cas62]{cassels4}
J.~W.~S. Cassels.
\newblock Arithmetic on curves of genus {$1$}. {IV}. {P}roof of the
  {H}auptvermutung.
\newblock {\em J. Reine Angew. Math.}, 211:95--112, 1962.

\bibitem[CDL16]{quaternion_linkage}
Adam Chapman, Andrew Dolphin, and Ahmed Laghribi.
\newblock Total linkage of quaternion algebras and {P}fister forms in
  characteristic two.
\newblock {\em J. Pure Appl. Algebra}, 220(11):3676--3691, 2016.

\bibitem[CRR13]{rapinchuk_genus_unramified}
Vladimir~I. Chernousov, Andrei~S. Rapinchuk, and Igor~A. Rapinchuk.
\newblock The genus of a division algebra and the unramified {B}rauer group.
\newblock {\em Bull. Math. Sci.}, 3(2):211--240, 2013.

\bibitem[CRR15]{rapinchuk_russian}
V.~I. Chernousov, A.~S. Rapinchuk, and I.~A. Rapinchuk.
\newblock Division algebras with the same maximal subfields.
\newblock {\em Uspekhi Mat. Nauk}, 70(1(421)):89--122, 2015.

\bibitem[CRR16]{size_genus}
Vladimir~I. Chernousov, Andrei~S. Rapinchuk, and Igor~A. Rapinchuk.
\newblock On the size of the genus of a division algebra.
\newblock {\em Tr. Mat. Inst. Steklova}, 292(Algebra, Geometriya i Teoriya
  Chisel):69--99, 2016.

\bibitem[CTP90]{parimala_real}
J.-L. Colliot-Th\'{e}l\`ene and R.~Parimala.
\newblock Real components of algebraic varieties and \'{e}tale cohomology.
\newblock {\em Invent. Math.}, 101(1):81--99, 1990.

\bibitem[CTPS12]{parimala_local_global}
Jean-Louis Colliot-Th\'{e}l\`ene, Raman Parimala, and Venapally Suresh.
\newblock Patching and local-global principles for homogeneous spaces over
  function fields of {$p$}-adic curves.
\newblock {\em Comment. Math. Helv.}, 87(4):1011--1033, 2012.

\bibitem[CTS21]{thelene_brauer}
Jean-Louis Colliot-Th\'el\`ene and Alexei~N. Skorobogatov.
\newblock {\em The {B}rauer-{G}rothendieck group}, volume~71 of {\em Ergebnisse
  der Mathematik und ihrer Grenzgebiete. 3. Folge. A Series of Modern Surveys
  in Mathematics [Results in Mathematics and Related Areas. 3rd Series. A
  Series of Modern Surveys in Mathematics]}.
\newblock Springer, Cham, [2021] \copyright2021.

\bibitem[FJ05]{field_arithmetic_book}
Michael~D. Fried and Moshe Jarden.
\newblock {\em Field arithmetic}, volume~11 of {\em Ergebnisse der Mathematik
  und ihrer Grenzgebiete. 3. Folge. A Series of Modern Surveys in Mathematics
  [Results in Mathematics and Related Areas. 3rd Series. A Series of Modern
  Surveys in Mathematics]}.
\newblock Springer-Verlag, Berlin, second edition, 2005.

\bibitem[Gro68]{brauer3}
Alexander Grothendieck.
\newblock Le groupe de {B}rauer. {III}. {E}xemples et compl\'{e}ments.
\newblock In {\em Dix expos\'{e}s sur la cohomologie des sch\'{e}mas}, volume~3
  of {\em Adv. Stud. Pure Math.}, pages 88--188. North-Holland, Amsterdam,
  1968.

\bibitem[Gro95]{brauer2}
Alexander Grothendieck.
\newblock Le groupe de {B}rauer. {II}. {T}h\'{e}orie cohomologique [
  {MR}0244270 (39 \#5586b)].
\newblock In {\em S\'{e}minaire {B}ourbaki, {V}ol. 9}, pages Exp. No. 297,
  287--307. Soc. Math. France, Paris, 1995.

\bibitem[GS06]{gille_sam}
Philippe Gille and Tam\'{a}s Szamuely.
\newblock {\em Central simple algebras and {G}alois cohomology}, volume 101 of
  {\em Cambridge Studies in Advanced Mathematics}.
\newblock Cambridge University Press, Cambridge, 2006.

\bibitem[GS10]{garibaldi_saltman}
Skip Garibaldi and David~J. Saltman.
\newblock Quaternion algebras with the same subfields.
\newblock In {\em Quadratic forms, linear algebraic groups, and cohomology},
  volume~18 of {\em Dev. Math.}, pages 225--238. Springer, New York, 2010.

\bibitem[JW90]{wadsworth_henselian}
Bill Jacob and Adrian Wadsworth.
\newblock Division algebras over {H}enselian fields.
\newblock {\em J. Algebra}, 128(1):126--179, 1990.

\bibitem[KM11]{gen_spl_krashen}
Daniel Krashen and Kelly McKinnie.
\newblock Distinguishing division algebras by finite splitting fields.
\newblock {\em Manuscripta Math.}, 134(1-2):171--182, 2011.

\bibitem[KMR{\etalchar{+}}22]{recent_genus}
Daniel Krashen, Eliyahu Matzri, Andrei Rapinchuk, Louis Rowen, and David
  Saltman.
\newblock Division algebras with common subfields.
\newblock {\em Manuscripta Math.}, 169(1-2):209--249, 2022.

\bibitem[Kuh11]{defect_kuhlmann}
Franz-Viktor Kuhlmann.
\newblock The defect.
\newblock In {\em Commutative algebra---{N}oetherian and non-{N}oetherian
  perspectives}, pages 277--318. Springer, New York, 2011.

\bibitem[Lan94]{lang_number_theory}
Serge Lang.
\newblock {\em Algebraic number theory}, volume 110 of {\em Graduate Texts in
  Mathematics}.
\newblock Springer-Verlag, New York, second edition, 1994.

\bibitem[Liu02]{liu}
Qing Liu.
\newblock {\em Algebraic geometry and arithmetic curves}, volume~6 of {\em
  Oxford Graduate Texts in Mathematics}.
\newblock Oxford University Press, Oxford, 2002.
\newblock Translated from the French by Reinie Ern\'{e}, Oxford Science
  Publications.

\bibitem[Mil20]{milneANT}
James~S. Milne.
\newblock Algebraic number theory (v3.08), 2020.
\newblock Available at www.jmilne.org/math/.

\bibitem[MM91]{mamm_merk}
P.~Mammone and A.~Merkurjev.
\newblock On the corestriction of {$p^n$}-symbol.
\newblock {\em Israel J. Math.}, 76(1-2):73--79, 1991.

\bibitem[Mor89]{morandi_henselization}
Pat Morandi.
\newblock The {H}enselization of a valued division algebra.
\newblock {\em J. Algebra}, 122(1):232--243, 1989.

\bibitem[Mor95]{morandi_defective}
Patrick~J. Morandi.
\newblock On defective division algebras.
\newblock In {\em {$K$}-theory and algebraic geometry: connections with
  quadratic forms and division algebras ({S}anta {B}arbara, {CA}, 1992)},
  volume~58 of {\em Proc. Sympos. Pure Math.}, pages 359--367. Amer. Math.
  Soc., Providence, RI, 1995.

\bibitem[Mor12]{morrow_higher_local}
Matthew Morrow.
\newblock An introduction to higher dimensional local fields and adeles, 2012.

\bibitem[Mou05]{karim_nsr}
Karim Mounirh.
\newblock Nicely semiramified division algebras over {H}enselian fields.
\newblock {\em Int. J. Math. Math. Sci.}, (4):571--577, 2005.

\bibitem[Pie82]{pierce_assoc}
Richard~S. Pierce.
\newblock {\em Associative algebras}, volume~9 of {\em Studies in the History
  of Modern Science}.
\newblock Springer-Verlag, New York-Berlin, 1982.

\bibitem[Poo17]{poonen_rational}
Bjorn Poonen.
\newblock {\em Rational points on varieties}, volume 186 of {\em Graduate
  Studies in Mathematics}.
\newblock American Mathematical Society, Providence, RI, 2017.

\bibitem[PR09]{pr2}
Gopal Prasad and Andrei~S. Rapinchuk.
\newblock Weakly commensurable arithmetic groups and isospectral locally
  symmetric spaces.
\newblock {\em Publ. Math. Inst. Hautes \'{E}tudes Sci.}, (109):113--184, 2009.

\bibitem[PR14]{pr1}
Gopal Prasad and Andrei~S. Rapinchuk.
\newblock Generic elements in {Z}ariski-dense subgroups and isospectral locally
  symmetric spaces.
\newblock In {\em Thin groups and superstrong approximation}, volume~61 of {\em
  Math. Sci. Res. Inst. Publ.}, pages 211--252. Cambridge Univ. Press,
  Cambridge, 2014.

\bibitem[PS96]{parimala_sujatha}
R.~Parimala and R.~Sujatha.
\newblock Hasse principle for {W}itt groups of function fields with special
  reference to elliptic curves.
\newblock {\em Duke Math. J.}, 85(3):555--582, 1996.
\newblock With an appendix by J.-L. Colliot-Th\'{e}l\`ene.

\bibitem[RR10]{RR_genus}
Andrei~S. Rapinchuk and Igor~A. Rapinchuk.
\newblock On division algebras having the same maximal subfields.
\newblock {\em Manuscripta Math.}, 132(3-4):273--293, 2010.

\bibitem[Sal99]{saltman_division}
David~J. Saltman.
\newblock {\em Lectures on division algebras}, volume~94 of {\em CBMS Regional
  Conference Series in Mathematics}.
\newblock Published by American Mathematical Society, Providence, RI; on behalf
  of Conference Board of the Mathematical Sciences, Washington, DC, 1999.

\bibitem[Ser79]{serre_local}
Jean-Pierre Serre.
\newblock {\em Local fields}, volume~67 of {\em Graduate Texts in Mathematics}.
\newblock Springer-Verlag, New York-Berlin, 1979.
\newblock Translated from the French by Marvin Jay Greenberg.

\bibitem[Sil09]{silverman_elliptic}
Joseph~H. Silverman.
\newblock {\em The arithmetic of elliptic curves}, volume 106 of {\em Graduate
  Texts in Mathematics}.
\newblock Springer, Dordrecht, second edition, 2009.

\bibitem[{Sta}22]{stacks-project}
The {Stacks project authors}.
\newblock The stacks project.
\newblock \url{https://stacks.math.columbia.edu}, 2022.

\bibitem[Tei37]{teichmuller}
Oswald Teichm\"{u}ller.
\newblock Zerfallende zyklische {$p$}-{A}lgebren.
\newblock {\em J. Reine Angew. Math.}, 176:157--160, 1937.

\bibitem[TW87]{tignol_wadsworth_ramified}
J.-P. Tignol and A.~R. Wadsworth.
\newblock Totally ramified valuations on finite-dimensional division algebras.
\newblock {\em Trans. Amer. Math. Soc.}, 302(1):223--250, 1987.

\bibitem[Wad02]{wadsworth}
A.~R. Wadsworth.
\newblock Valuation theory on finite dimensional division algebras.
\newblock In {\em Valuation theory and its applications, {V}ol. {I}
  ({S}askatoon, {SK}, 1999)}, volume~32 of {\em Fields Inst. Commun.}, pages
  385--449. Amer. Math. Soc., Providence, RI, 2002.

\bibitem[Wit37]{witt}
Ernst Witt.
\newblock Zyklische {K}\"{o}rper und {A}lgebren der {C}harakteristik {$p$} vom
  {G}rad {$p^n$}. {S}truktur diskret bewerteter perfekter {K}\"{o}rper mit
  vollkommenem {R}estklassenk\"{o}rper der {C}harakteristik {$p$}.
\newblock {\em J. Reine Angew. Math.}, 176:126--140, 1937.

\bibitem[Zhu00]{higher_local}
Igor Zhukov.
\newblock Higher dimensional local fields.
\newblock In {\em Invitation to higher local fields ({M}\"{u}nster, 1999)},
  volume~3 of {\em Geom. Topol. Monogr.}, pages 5--18. Geom. Topol. Publ.,
  Coventry, 2000.

\bibitem[Če19]{purity_brauer}
Kęstutis Česnavičius.
\newblock {Purity for the Brauer group}.
\newblock {\em Duke Mathematical Journal}, 168(8):1461 -- 1486, 2019.

\end{thebibliography}
\end{document}